\documentclass{kybernetika} 

\allowdisplaybreaks

\newtheorem{thm}{Theorem}[section]
\newtheorem{lem}[thm]{Lemma}
\newtheorem{prop}[thm]{Proposition}
\newtheorem{cor}[thm]{Corollary}

\newtheorem{dfn}[thm]{Definition}

\newtheorem{ques}[thm]{Question}
\newtheorem{ex}[thm]{Example}
\newtheorem{exs}[thm]{Examples}
\newtheorem{rmk}[thm]{Remark}
\newtheorem{rmks}[thm]{Remarks}

\renewcommand{\rm}[1]{\mathrm{#1}}
\renewcommand{\cal}[1]{\mathcal{#1}}

\newcommand{\bbE}{\boldsymbol{\mathsf{E}}}

\newcommand{\bbN}{\mathbb{N}}

\newcommand{\bbR}{\mathbb{R}}

\newcommand{\bbZ}{\mathbb{Z}}

\newcommand{\rmH}{\mathrm{H}}

\renewcommand{\d}{\mathrm{d}}
\newcommand{\rme}{\mathrm{e}}
\newcommand{\rmh}{\mathrm{h}}

\newcommand{\B}{\mathcal{B}}

\newcommand{\F}{\mathcal{F}}

\renewcommand{\P}{\mathcal{P}}
\newcommand{\Q}{\mathcal{Q}}
\newcommand{\R}{\mathcal{R}}

\renewcommand{\O}{\Omega}

\newcommand{\eps}{\varepsilon}
\newcommand{\g}{\gamma}
\renewcommand{\l}{\lambda}
\renewcommand{\k}{\kappa}
\renewcommand{\o}{\omega}
\newcommand{\s}{\sigma}

\newcommand{\ol}[1]{\overline{#1}}

\newcommand{\fin}{\nolinebreak\hspace{\stretch{1}}$\lhd$}

\renewcommand{\to}{\longrightarrow}

\renewcommand{\phi}{\varphi}

\renewcommand{\Pr}{\rm{Prob}}

\newcommand{\TC}{\mathrm{TC}}
\newcommand{\DTC}{\mathrm{DTC}}

\newcommand{\rmD}{\mathrm{D}}
\newcommand{\rmI}{\mathrm{I}}

\begin{document}
\pagestyle{myheadings}

\title{Multi-variate correlation and\newline mixtures of product measures}

\author{Tim Austin}

\contact{Tim}{Austin}{UCLA Mathematics Department, Box 951555, Los Angeles CA 90095. U.\,S.\,A.}{tim@math.ucla.edu}

\markboth{T. Austin} {Correlation and mixtures of products}

\maketitle

\begin{abstract}
Total correlation (`TC') and dual total correlation (`DTC') are two classical ways to quantify the correlation among an $n$-tuple of random variables.  They both reduce to mutual information when $n=2$.

The first part of this paper sets up the theory of TC and DTC for general random variables, not necessarily finite-valued.  This generality has not been exposed in the literature before.

The second part considers the structural implications when a joint distribution $\mu$ has small TC or DTC.  If $\mathrm{TC}(\mu) = o(n)$, then $\mu$ is close to a product measure according to a suitable transportation metric: this follows directly from Marton's classical transportation-entropy inequality.  If $\mathrm{DTC}(\mu) = o(n)$, then the structural consequence is more complicated: $\mu$ is a mixture of a controlled number of terms, most of them close to product measures in the transportation metric.  This is the main new result of the paper.
\end{abstract}



\setcounter{tocdepth}{1}


\section{Introduction}

Let $X_1,\dots,X_n$ be an $n$-tuple of random variables.  When $n=2$, the mutual information $\rmI(X_1\,;\,X_2)$ is a canonical way to quantify the dependence between them.  Once $n\geq 3$, mutual information can be generalized in several different ways, suitable for different purposes.

This paper focuses on two of these.  If $X_1$, \dots, $X_n$ are finite-valued, then their `total correlation' (`TC') is
\[\TC(X_1\,;\,\dots\,;\,X_n) = \Big(\sum_{i=1}^n\rmH(X_i)\Big) - \rmH(X_1,\dots,X_n),\]
and their `dual total correlation' (`DTC') is
\[\DTC(X_1\,;\,\dots\,;\,X_n) := \rmH(X_1,\dots,X_n) - \sum_{i=1}^n \rmH(X_i\,|\,X_1,\dots,X_{i-1},X_{i+1},\dots,X_n).\]
These definitions are discussed more carefully in Subsection~\ref{subs:dfns}, and extended to general random variables as suprema over quantizations.

Part I of the present paper builds up the basic theory of TC and DTC.  The main results in this part are some identities and inequalities relating TC, DTC, Shannon entropy, and mutual information. These can mostly be seen as generalizations of the chain rule and monotonicity properties of mutual information.  We take care to define and study TC and DTC for general random variables, not just finite-valued ones, and this introduces various additional technicalities.  In handling these technicalities we extend foundational work of Kolmogorov, Dobrushin, Gelfand, Yaglom and Perez on mutual information.

The proofs in Part I are quite routine: a large fraction of the work boils down to applications of the chain rule.  Many of these calculations have been done before.  I include the references I know, but suspect that more lie buried in the literature.

Both TC and DTC are non-negative, and zero only in the case of independent random variables.  One can therefore look for stability results for these quantities: does a small value of TC or DTC imply some `structure' which is close to independence?  This is the topic of Part II of the paper.

In our main results of this kind, `closeness' of distributions is in the sense of transportation metrics.  Suppose that each $X_i$ takes values in a complete and separable metric space $(K_i,d_{K_i})$, and assume that each $d_{K_i}$ has diameter at most $1$.  We endow the product $\prod_i K_i$ with the normalized Hamming average of these metrics:
\begin{equation}\label{eq:Hamming}
d_n(x,y) := \frac{1}{n}\sum_{i=1}^n d_{K_i}(x_i,y_i) \qquad \Big(x,y \in \prod_iK_i\Big).
\end{equation}
If each $K_i$ is just a finite set endowed with the discrete metric, then $d_n$ is the normalized Hamming metric on $\prod_i K_i$.  This case already involves most of the ideas that we need.  The case of other metrics on finite sets is easily reduced to this one, since then each $d_{K_i}$ is bounded above by the discrete metric on $K_i$, and the results for the latter imply those for the former.  However, if one of the spaces $K_i$ is uncountable, then switching from $d_{K_i}$ to the discrete metric sacrifices separability and introduces new Borel sets on which our measures may not be defined.  For the sake of this case we formulate and prove our results for general metrics of diameter at most $1$.

Using $d_n$, we endow the set $\Pr(\prod_iK_i)$ of all probability distributions with the transportation metric:
\begin{equation}\label{eq:trans-met}
\ol{d_n}(\mu,\nu) := \inf_\l \int d_n(x,y)\,\l(\d x,\d y) \qquad \Big(\mu,\nu \in \Pr\Big(\prod_iK_i\Big)\Big),
\end{equation}
where $\l$ ranges over all couplings of $\mu$ and $\nu$.

For TC, existing results can easily be re-cast as a structural conclusion of the desired kind: if $\TC = o(n)$, then the joint distribution $\mu$ of $X_1$, \dots, $X_n$ is close in $\ol{d_n}$ to a product measure.  This is a corollary of Marton's classical transportation-entropy inequality~\cite{Mar86,Mar96}: see Proposition~\ref{prop:Marton} and Corollary~\ref{cor:Marton} below.

For DTC the situation is more complicated.  We always have
\[\TC \leq (n-1)\cdot \DTC\]
(see Lemma~\ref{lem:basic-ineqs}), so the result mentioned above for TC can be applied if $\DTC = o(1)$.  Outside this range, $\mu$ need not be close to a product: indeed, any mixture of $k$ product measures always has DTC at most $\log k$ (see Proposition~\ref{prop:DTC-approx-concav}).

However, it turns out that this is roughly the only possibility over a much larger range of DTC-values: if $\DTC = o(n)$, then $\mu$ is close in $\ol{d_n}$ to a `low-complexity' mixture of product measures.  This is the main result of the paper.

\paragraph{Theorem A.}
	Fix a parameter $\delta > 0$ and let $\mu \in \Pr(\prod_i K_i)$.  If $\DTC(\mu) \leq \delta^3 n$, then $\mu$ may be written as a mixture
	\[\mu = \int_L \mu_y\,\nu(\d y)\]
so that
	\begin{enumerate}
		\item[(a)] the mutual information in the mixture satisfies $\rmI(\nu,\mu_\bullet) \leq \DTC(\mu)$, and
		\item[(b)] there is a measurable family $(\xi_y:\ y \in L)$ of product measures on $\prod_i K_i$ such that
		\[\int_L \ol{d_n}(\mu_y,\xi_y)\,\nu(\d y) < 2\delta.\]
	\end{enumerate}

In this statement, the `mutual information in the mixture' $\rmI(\nu,\mu_\bullet)$ is the mutual information between $(X_1,\dots,X_n) \sim \mu$ and the mixing parameter $y \sim \nu$ when they are coupled by the kernel ${(\mu_y:\ y\in L)}$: see Subsection~\ref{subs:mix-mut-inf}.

In the proof of Theorem A, the mixture representing $\mu$ is obtained by conditioning on a small subset $S$ of the coordinates in $\{1,2,\dots,n\}$, resulting in the explicit choice $L := \prod_{i \in S}K_i$.

The bounds provided by Theorem A do not depend on the alphabets $K_i$ at all.  However, Theorem A has variants in which such a dependence does appear.  For instance, Theorem~\ref{thm:A'} below gives a nontrivial conclusion in the range $\DTC(\mu) \ll \delta^2 n$, so it is slightly less restrictive than Theorem A, but the alphabets $K_i$ must be finite.  It gives an alternative to part (a) of Theorem A which depends on the cardinalities $|K_i|$.  See Subsection~\ref{subs:number-of-terms} for more discussion.

Because our main structural results give a description only up to approximations in $\ol{d_n}$, they are not completely satisfactory.  For TC, the results described above give:
\begin{align*}
&\hbox{product measure} \\
&\quad \Longrightarrow \qquad \TC = 0  \qquad \hbox{(standard calculation)}\\
&\quad \qquad \Longrightarrow \qquad \hbox{small TC}  \qquad \hbox{(trivially)}\\
&\quad \qquad \qquad \Longrightarrow \qquad \hbox{\emph{near}-product measure} \qquad \hbox{(Corollary~\ref{cor:Marton})},
\end{align*}
where `nearness' refers to an approximation in $\ol{d_n}$. This leaves open the possibility that some near-product measures have small TC while others have large TC.  In fact we can say a little more about this gap, because if each $(K_i,d_{K_i})$ is a finite set with its discrete metric then TC enjoys some `continuity' in the metric $\ol{d_n}$.  However, the quality of this continuity deteriorates as the cardinalities $|K_i|$ increase: see Lemma~\ref{lem:TC-robust}.  So, if each $K_i$ is finite, then any measure sufficiently close to a product must have small TC, but the necessary closeness here also depends on those cardinalities. This still leaves a gap in our understanding when some of the values $|K_i|$ are large or infinite, and indeed there are measures which occupy that gap and have large TC: see Example~\ref{ex:near-prod-big-TC}.

Similarly, for DTC we have
\begin{align*}
&\hbox{low-complexity mixture of products} \\
&\quad \Longrightarrow \qquad \hbox{small DTC}  \qquad \hbox{(Proposition~\ref{prop:DTC-approx-concav})}\\
&\quad \qquad \Longrightarrow \qquad \hbox{low-complexity mixture of \emph{near}-products} \qquad \hbox{(Theorem A)}.
\end{align*}
But it remains unclear when a mixture of near-products has small or large DTC.  Unlike for TC, there seems to be no useful `continuity' for DTC at all: Example~\ref{ex:near-prod-big-DTC} has transportation distance less than $1/n$ from a product measure, but has large DTC.  

\subsection*{Relation to previous work}

The study of measures of multi-variate correlation began with McGill's notion of `interaction information'~\cite{McG54}.  Since then, numerous other proposals have been studied in the information theory literature, and numerous identities relating them have been uncovered.  TC was first studied by Watanabe in~\cite{Wat60}, and DTC by Han in~\cite{Han75}.  Confusingly, both of these quantities and several of the others seem to have been rediscovered multiple times, and given a new name each time.  Here we use Watanabe's and Han's original names.

Basic theoretical work in this area has been driven by a search for ways to identify different kinds of dependence structure among several random variables: for instance, disjoint subcollections of them exhibiting some conditional independence.  See~\cite{PeaPaz87,Per77} for some early analyses along these lines, and~\cite{StuVej98} for a more recent overview.

Much of the literature on notions of multi-variate correlation concerns their application in other branches of science.  The paper~\cite{TimAlfFleBeg14} recalls several of these notions and discusses the practical matter of choosing one for the sake of interpreting different kinds of experimental data.  The note~\cite{Crooks--mutinfnote} gives a quicker survey of the many options.  Both of these references contain a more complete guide to the literature.  A concrete example of an application of TC can be found in~\cite[Example 4]{CsiNar04}, where it provides an upper bound on the secret key capacity in a certain network communication model of secrecy generation.

TC and DTC retain one basic property of mutual information which is especially relevant to our work: if $X_i$ determines $Y_i$ for each $i$, then
\[\TC(X_1\,;\,\dots\,;\,X_n) \geq \TC(Y_1\,;\,\dots\,;\,Y_n),\]
and similarly with DTC in place of TC.  In particular, TC and DTC are both non-negative. This non-negativity is included in a more general family of inequalities discovered by Han~\cite{Han78}; see~\cite[Theorem 17.6.1]{CovTho06} for a textbook treatment. Some other generalizations of mutual information, including McGill's interaction information and several others listed in~\cite{Crooks--mutinfnote,TimAlfFleBeg14}, lose this monotonicity property.  They can vanish for quite highly correlated random variables, depending on the nature of the correlation.

However, TC and DTC are not the only non-negative measures of multi-variate correlation.  On the contrary, they can be seen as the two simplest members of a large family, all obtained as the gaps in different entropy inequalities, and all non-negative as a result.  These include the other inequalities from~\cite{Han78}, which in turn are all special cases of an inequality due to Shearer (also from about 1978, but first published in~\cite{ChungGraFraShe86}). Fujishige interpreted Han's inequalities using polymatroids in~\cite{Fuj78}, and explicitly investigated possible values for the gaps. Further refinements of Shearer's inequality have recently been investigated by Madiman and Tetali~\cite{MadTet10} and Balister and Bollob\'as~\cite{BalBol12}.  We restrict attention to TC and DTC in most of this paper, but return briefly to those other inequalities in the final section.

Beyond monotonicity, our specific interest in TC and DTC stems from previous structural results about those quantities.  The TC of a joint distribution $\mu$ is equal to the Kullback--Leibler divergence between $\mu$ and the product of its marginals (see equation~\eqref{eq:TC-and-D} below).  In this guise, it appears in Csisz\'ar's paper~\cite{Csi84} on generalizing Sanov's theorem and conditional limit theorems, and is implicitly at work in Marton's concentration inequality for product measures (see Section~\ref{sec:TC-and-prod} below).

DTC has received less theoretical attention than TC.  However, in the recent paper~\cite{Aus--WP} it plays a crucial role in a new decomposition theorem for measures on product spaces, of a similar flavour to Theorem A but with a different conclusion.

In~\cite{Aus--WP}, a measure on a product space is decomposed into a mixture whose terms mostly satisfy a kind of concentration inequality called a T-inequality.  Following the terminology in~\cite{Aus--WP}, a measure $\mu
\in \Pr(\prod_i K_i)$ satisfies T$(a,r)$ for some parameters $a,r > 0$ if we have
\begin{equation}\label{eq:T-ineq}
\int \rme^{a f}\,\d\mu \leq \exp\Big(a \int f\,\d\mu + a r\Big)
\end{equation}
whenever $f:\prod_i K_i \to \bbR$ is $1$-Lipschitz for the metric $d_n$.  (To be precise, the notation T$(a,r)$ is used in~\cite{Aus--WP} for a certain transportation-entropy inequality, which is then proven equivalent to~\eqref{eq:T-ineq} via a version of the Bobkov--G\"otze equivalence: see~\cite[Subsection 5.3]{Aus--WP}.)

The decomposition result in~\cite{Aus--WP} is initially formulated for TC (see~\cite[Theorem B]{Aus--WP}) in order to meet the needs of an application to ergodic theory. But that version is derived from the analogous result for DTC~\cite[Theorem 7.1]{Aus--WP}.  Here are those results:

\begin{thm}\label{thm:from-WP}
 For any $\eps,r > 0$ there exist $c,\kappa > 0$ for which the following holds.  Any $\mu \in \Pr(\prod_i K_i)$ can be written as a mixture
\[\mu = p_1\mu_1 + \cdots + p_m\mu_m\]
so that
\begin{itemize}
	\item[(a)] $m\leq c\cdot \exp(c\cdot \TC(\mu))$,
	\item[(b)] $p_1 < \eps$, and
	\item[(c)] the measure $\mu_j$ satisfies T$(\k n,r)$ for each $j=2,3,\dots,m$.
	\end{itemize}
The same conclusion holds with DTC in place of TC, except the dependence of $c$ and $\kappa$ on $\eps$ and $r$ may be different.
\end{thm}

  The switch from TC to DTC is an essential idea in the proof of Theorem~\ref{thm:from-WP} in~\cite{Aus--WP}. DTC exhibits a `decrement' under a certain splitting operation on measures, and this phenomenon is central to the proof of Theorem~\ref{thm:from-WP}. It seems to have no analog for TC.  See~\cite[Sections 6 and 7]{Aus--WP}.

  This ability of DTC to find the structure in Theorem~\ref{thm:from-WP} was one of the main discoveries in~\cite{Aus--WP}.  The present paper is largely motivated by that discovery, although the proof of Theorem A above is quite different from the arguments in~\cite{Aus--WP}, and much shorter. Aside from Theorem~\ref{thm:from-WP}, the literature contains few stability or structural results based on DTC (another possible example is the main result of~\cite{EllFriKinYeh16}).  The present paper fills in some more of this picture.

Any product measure on $\prod_i K_i$ satisfies T$(8rn, r)$ for all $r > 0$.  This is essentially McDiarmid's inequality~\cite{McD89}, and is equivalent to Marton's transportation-entropy inequality (Proposition~\ref{prop:Marton} below) via the Bobkov--G\"otze equivalence.  In addition, perturbations in $\ol{d_n}$ preserve T-inequalities up to some fiddly trimming and a slight deterioration in the constants (see~\cite[Proposition 5.9]{Aus--WP}), so near-product measures still satisfy fairly good T-inequalities up to that trimming.  But many measures that are far from any product measure in $\ol{d_n}$ also satisfy good T-inequalities: see~\cite[Subsection 5.2]{Aus--WP} for examples and more discussion.  For this reason, the T-inequalities promised by part (c) of Theorem~\ref{thm:from-WP} are strictly weaker than the `near-product' structure obtained in part (b) of Theorem~A.

On the other hand, Theorem~\ref{thm:from-WP} applies for any value of TC or DTC, and part (a) of that theorem gives a straightforward exponential bound in this parameter. These features are crucial for the application of Theorem~\ref{thm:from-WP} in~\cite{Aus--WP}.  By contrast, the results of the present paper give nothing at all unless TC or DTC is sufficiently small compared to $n$.  When TC and DTC are too large for the results of the present paper to apply, Theorem~\ref{thm:from-WP} remains applicable and possibly valuable.  A relevant example is given in~\cite[Example 5.4]{Aus--WP}. It is a measure on $A^n$ for a large finite alphabet $A$ which (i) has both TC and DTC of order $n$, but small compared to $n\log |A|$, (ii) already satisfies a good T-inequality, but (iii) cannot be written as a mixture of near-product measures with any meaningful control on the complexity of the mixture.  Clearly Theorem A of the present paper cannot be extended to cover that example.

Let us describe the choice between these different structural results more quantitatively in case each $K_i$ is finite, say of size at most $k$. For TC, Theorem~\ref{thm:from-WP} is non-trivial when $\TC = o(n\log k)$ --- if TC is larger than this, then we might as well partition $\prod_i K_i$ into singletons.  But the simpler result of Corollary~\ref{cor:Marton} below takes over when $\TC = o(n)$.  As a result, the TC-part of Theorem~\ref{thm:from-WP} is really valuable only when $k$ is large.  This was already remarked in~\cite{Aus--WP} following the statement of~\cite[Theorem B]{Aus--WP}.

The analogous discussion for DTC was left incomplete in~\cite{Aus--WP}.  Theorem A from the present paper completes it, up to approximations in $\ol{d_n}$.  As above, Theorem~\ref{thm:from-WP} is non-trivial when $\DTC = o(n \log k)$, but now Theorem A takes over when $\DTC = o(n)$.  As for TC, this means that the DTC-part of Theorem~\ref{thm:from-WP} is really valuable only when $k$ is large.

These different structural results from~\cite{Aus--WP} and the present paper are summarized in this table:

\vspace{7pt}

\begin{tabular}{|l|p{200pt}|}
	\hline
	Range where nontrivial & Conclusion\\ \hline
	$\TC = o(n)$ & $\ol{d_n}$-close to a product\\ \hline
	$\phantom{\TC} = o(n\log k)$ & mixture with mut.inf. $= O(\TC)$, most $\phantom{booooo}$ terms concentrated\\ \hline
	$\DTC = o(1)$ & $\ol{d_n}$-close to a product\\ \hline
	$\phantom{\DTC} = o(n)$ & mixture with mut.inf. $= O(\DTC)$, most $\phantom{booooo}$ terms $\ol{d_n}$-close to products\\ \hline
	$\phantom{\DTC} = o(n\log k)$ & mixture with mut.inf. $= O(\DTC)$, most $\phantom{booooo}$ terms concentrated\\
	\hline
\end{tabular}

\subsection*{Relation to nonlinear large deviations}

Several recent papers in probability have begun to develop a new theory of `nonlinear large deviations', starting with Chatterjee and Dembo's work in~\cite{ChaDem16}.  This theory includes examples of measures for which Theorem A is nontrivial.

Let $(X_1,\dots,X_n)$ be a tuple of independent random variables, say taking values in $K_1\times \cdots \times K_n$ and having joint distribution $\l_1\times \cdots\times \l_n$. This new theory considers functions $f(X_1,\dots,X_n)$ which are more complicated than sums of functions of the individual coordinates (this is the meaning of `nonlinear'), and seeks conditions on $f$ under which one can still obtain good estimates on probabilities of large deviations of $f(X_1,\dots,X_n)$ from its mean.  Problems of this kind occur naturally in the study of triangle- or other subgraph-counts in Erd\H{o}s--R\'enyi random graphs, and in questions of a similar flavour about arithmetic progressions in random arithmetic sets.

After suitably modifying the function $f$, this large deviations problem can be translated into that of estimating the partition function $Z = \int \rme^f\,\d(\l_1\times \cdots \times \l_n)$ that normalizes the Gibbs measure
\[\mu(\d x) := \frac{\rme^{f(x)}}{Z}\l_1(\d x_1)\cdots \l_n(\d x_n).\]
In~\cite{ChaDem16}, Chatterjee and Dembo give a good approximation to $Z$ when $K_i = \{0,1\}$, $\l_i$ is uniform for each $i$, and $f$ satisfies a `low complexity' assumption on its `discrete gradients'.  When $K_i = \{0,1\}$ for each $i$, the discrete gradient $\nabla_xf$ at a point $x \in \{0,1\}^n$ is the vector $(\partial_i f(x))_{i=1}^n$, where
\[\partial_i f(x) = f(x_1,\dots,x_{i-1},1,x_{i+1},\dots,x_n) - f(x_1,\dots,x_{i-1},0,x_{i+1},\dots,x_n).\]
The function $f$ is linear if and only if $\nabla_x f$ is the same vector for all $x$.  Beyond this case, Chatterjee and Dembo consider functions $f$ for which the set of discrete gradients ${\{\nabla_x f:\ x \in \{0,1\}^n\}}$ can be covered by relatively few (intuitively, $\exp(o(n))$) small-radius balls in $\bbR^n$ --- this is the meaning `low complexity'.  Chatterjee and Dembo's results are extended to more general product spaces in~\cite{Yan--nldp}.

Following the insight that this `low complexity' assumption enables good approximations to $Z$, Eldan showed that it also implies a relatively simple structure for the Gibbs measure $\mu$.  In~\cite{Eld--GWcomp} he considers the product space $\{-1,1\}^n$ (technically more convenient than $\{0,1\}^n$ for his proofs), assumes that ${\{\nabla_x f:\ x\in\{-1,1\}^n\}}$ is small in a slightly stronger sense given by the notion of `Gaussian width', and deduces that $\mu$ is a mixture of near-product measures with relatively small mutual information in the mixture.  His later papers~\cite{EldGro--exprndmgraphs,EldGro18} with Gross give a more precise description of \emph{which} products measures appear in the mixture, with even finer results in the setting of exponential random graph models.

So Eldan's structural conclusion is similar to Theorem A, and it is natural to wonder whether these `low complexity' assumptions on $f$ imply a bound on the DTC of the Gibbs measure $\mu$.  It turns out that they do, even under Chatterjee and Dembo's original and weaker covering-number assumption, and in a way that can easily be formulated over arbitrary product spaces.  This is proved in the preprint~\cite{Aus--low-cplx-Gibbs},

The resulting bound on DTC implies a structural conclusion similar to Eldan's.  But this implication is somewhat misleading.  In fact, the bound on DTC obtained in the main result of~\cite{Aus--low-cplx-Gibbs} is a by-product of another, more direct proof that $\mu$ is a mixture of near-products with low mutual information.  That more direct proof uses basic information theoretic principles such as Gibbs' variational principle and Marton's transportation-entropy inequality.  It applies over arbitrary product spaces once one has the right general definition of discrete gradients.  It gives better estimates on the structure of $\mu$ than one would obtain by taking the resulting bound on DTC and then applying Theorem A.  So~\cite{Aus--low-cplx-Gibbs} actually takes very little from the present paper: in addition to the non-negativity of DTC (Han's inequality), it just needs the alternative formula for DTC that we prove below in part (b) of Proposition~\ref{prop:TC-DTC-and-D}.  But it does treat a large class of probabilistic models that illustrate the structure appearing in Theorem A.
\\[4mm]

\noindent{\bf \large Part I: Some basic theory of multi-variate correlation}


\section{Background from probability}

\subsection{Basic conventions}

We assume standard results and notation from measure-theoretic probability.  In Part II we restrict attention to standard Borel probability spaces, in order to use disintegrations freely. We often denote a measurable space by a single letter such as $K$.  If it is standard Borel, then we denote its sigma-algebra by $\B_K$, and we write $\Pr(K)$ for the convex set of Borel probability measures on it.  If $K$ is a complete and separable metric space, then it is endowed with its Borel sigma-algebra by default.

If $X$ is a random variable taking values in a measurable space $K$, and $\P$ is a finite measurable partition of $K$, then we write $[X]_\P$ for the \textbf{quantization} of $X$ by $\P$: the $\P$-valued random variable defined by
\[[X]_\P = P \quad \Longleftrightarrow \quad X \in P.\]
Similarly, if $\mu$ is a probability measure on $X$, then $[\mu]_\P$ is the $\P$-indexed stochastic vector $(\mu(P):\ P\in\P)$.

Whenever several random variables are under discussion simultaneously, they are defined on the same underlying probability space.

\subsection{Kernels, mixtures and randomizations}

Let $(\O,\F)$ be a measurable space and $K$ a standard Borel space.  We assume familiarity with the notion of a kernel from $\O$ to $K$.  We usually denote such a kernel by $\o \mapsto \mu_\o$ or something similar.  Given such a kernel and also a probability measure $P$ on $\O$, the resulting \textbf{mixture} is the measure $\mu$ on $K$ defined by
\begin{equation}\label{eq:mix}
\mu(A) := \int_\O \mu_\o(A)\,P(\d \o) \quad \forall A \in \B_K.
\end{equation}

If $\mu_\bullet$ and $P$ are as above, then they also define a measure on the product space $(\O\times K,\F\otimes \B_K)$: the measure of $E \in \F\otimes \B_K$ is
\[\int_\O \mu_\o\{x:\ (\o,x) \in E\}\,P(\d \o).\]
The correctness of this definition is a standard extension of the proof of Fubini's theorem: see, for instance,~\cite[Theorem 10.2.1(II)]{Dud--book}.  We denote this new measure by $P \ltimes \mu_\bullet$.  It may also be described as the mixture obtained from the measure $P$ and the $(\O\times K)$-valued kernel $\o \mapsto \delta_\o\times \mu_\o$.  The marginal of $P\ltimes \mu_\bullet$ on the $K$-coordinate is precisely the mixture~\eqref{eq:mix}.

Since $K$ is standard Borel, any probability measure $\l$ on $(\O\times K,\F\otimes \B_K)$ can be written as $P\ltimes \mu_\bullet$ in an essentially unique way.  In this representation, $P$ is simply the marginal of $\l$ on $\O$, and then $\mu_\o$ is unique up to agreement for $P$-almost every $\o$.  These assertions are the existence and uniqueness parts of the measure disintegration theorem: see, for instance,~\cite[Theorem 10.2.2]{Dud--book}.

When $\mu \in \Pr(K)$ is represented by the mixture~\eqref{eq:mix}, a \textbf{randomization} of that mixture is any pair of random variables $(Y,X)$ on some background probability space whose joint distribution is $P \ltimes \mu_\bullet$ (so $Y$ takes values in $\O$ and $X$ takes values in $K$).  We invoke randomizations a few times below, because some information theoretic quantities are more easily described in terms of random variables, even though they depend only on the distributions of those random variables.

\section{Background from information theory}

\subsection{Finite-valued random variables}

Let $X$, $Y$ and $Z$ be finite-valued random variables on a common probability space.  We assume familiarity with the Shannon entropy $\rmH(X)$ and mutual information
\begin{equation}\label{eq:mut-inf}
\rmI(X\,;\,Y) = \rmH(X) + \rmH(Y) - \rmH(X,Y);
\end{equation}
with their conditional versions $\rmH(X\,|\,Z)$ and $\rmI(X\,;\,Y\,|\,Z)$; and with the \textbf{chain rules} that these quantities satisfy:
\[\rmH(X,Z) := \rmH(Z) + \rmH(X\,|\,Z) \quad \hbox{and} \quad \rmI(X\,;\,Y,Z) = \rmI(X\,;\,Z) + \rmI(X\,;\,Y\,|\,Z).\]
See, for instance,~\cite[Chapter 2]{CovTho06}.

\subsection{Extension to general random variables}\label{subs:basic-gen-RVs}

The definitions of $\rmH(X\,|\,Z)$ and $\rmI(X\,;\,Y\,|\,Z)$ are easily extended to the case of an arbitrary random variable $Z$, provided $X$ and $Y$ are still finite-valued.  For Shannon entropy, we have
\begin{equation}\label{eq:H-gen-dfn}
\rmH(X\,|\,Z) := \inf_\P \rmH(X\,|\,[Z]_\P),
\end{equation}
where the infimum runs over all finite quantizations of $Z$.  Then we define
\begin{equation}\label{eq:pre-cond-mut-inf}
\rmI(X\,;\,Y\,|\,Z) := \rmH(X\,|\,Z) - \rmH(X\,|\,Y,Z)
\end{equation}
as before.

Mutual information can be defined for a general pair of random variables $X$ and $Y$ by quantizing:
\begin{equation}\label{eq:I-gen-dfn}
\rmI(X\,;\,Y)  = \sup_{\P,\Q}\rmI([X]_\P\,;\,[Y]_\Q),
\end{equation}
where the supremum runs over all finite quantizations of $X$ and $Y$.  This mutual information may still be finite for non-discrete random variables: indeed, it is zero whenever $X$ and $Y$ are independent.

The analysis of mutual information for general random variables began with Kolmogorov~\cite{Kol56} and Dobrushin~\cite{Dob59}.  It is recounted carefully in Pinsker's classic book~\cite{Pin64}.  Some of their proofs require that the random variables take values in standard Borel spaces (see the translator's remarks to~\cite[Chapter 3]{Pin64}).  One can avoid this assumption by working entirely through quantizations, following Gelfand, Kolmogorov, Yaglom and Perez~\cite{GelKolYag56,Per59}.  We include a quick account here for reference.

The infimum in~\eqref{eq:H-gen-dfn} and supremum in~\eqref{eq:I-gen-dfn} are monotone under refinement of the partitions $\P$ and $\Q$: non-increasing in~\eqref{eq:H-gen-dfn} and non-decreasing in~\eqref{eq:I-gen-dfn}.  This fact is crucial to the understanding of these quantities.  It allows us to think of these infima and suprema as limits under refinement of partitions, and this point of view simplifies several proofs.  We meet similar situations later when we study TC and DTC.

It is also useful to know that these infima and suprema can be restricted to special families of partitions.

The following terminology helps us to formulate these properties precisely.

\begin{dfn}\label{dfn:adequate}
 For any measurable space $K$, a family $\frak{P}$ of finite measurable partitions of $K$ is
	\begin{itemize}
		\item \textbf{directed} if any two members of $\frak{P}$ have a common refinement in $\frak{P}$;
		\item \textbf{generating} if together the members of $\frak{P}$ generate the whole sigma-algebra of $K$.
	\end{itemize}
\end{dfn}

\begin{exs}\label{exs:adequate}
	\begin{enumerate}
		\item The family of all finite measurable partitions of $K$ is clearly directed and generating.
		\item Suppose that $K$ is a product $K_1\times K_2$ of two other measurable spaces with the product sigma-algebra, and that $\frak{P}_i$ is a family of finite partitions of $K_i$ for $i=1,2$.  Whenever $\P_i \in \frak{P}_i$ for $i=1,2$, we can identify $\P_1\times \P_2$ with the family
		\[\{A\times B:\ A \in \P_1,\ B \in \P_2\},\]
		which is a partition of $K$ into measurable rectangles.  If each $\frak{P}_i$ is directed (respectively, generating), then the collection $\{\P_1\times \P_2:\ \P_i \in \frak{P}_i\}$ is directed (respectively, generating) in $K$.
		
		This example generalizes directly to larger finite products.
		
		\item If the sigma-algebra of $K$ is countably generated, then it has a generating filtration
		\[\{\emptyset,K\}\subseteq \F_1\subseteq \F_2 \subseteq \dots\]
		consisting of finite algebras of subsets.  Let $\P_i$ be the partition of $K$ into the atoms of $\F_i$ for each $i$.  Then the collection $\frak{P} := \{\P_i:\ i\geq 1\}$ is directed and generating.
		\end{enumerate}
	\end{exs}

We can now formulate the notion of `limit' that we need.

\begin{dfn}\label{dfn:adequate-conv}
 Let $\frak{P}_1$, \dots, $\frak{P}_n$ be directed families of partitions of the measurable spaces $K_1$, \dots, $K_n$, respectively, and let $\phi$ be a function from $\frak{P}_1\times \cdots \times \frak{P}_n$ to the extended real line $[-\infty,\infty]$.  Let $c \in [-\infty,\infty]$. We write
	\[\lim_{\P_1\in \frak{P}_1,\P_2 \in \frak{P}_2,\dots ,\P_n \in \frak{P}_n}\phi(\P_1,\dots,\P_n) = c\]
to mean that, for every neighbourhood $U$ of $c$, there exist $\Q_1 \in \frak{P}_1$, \dots, $\Q_n \in \frak{P}_n$ such that
	\[\phi(\P_1,\dots,\P_n) \in U\]
	whenever $\P_i \in \frak{P}_i$ and $\P_i$ refines $\Q_i$ for each $i$.
	\end{dfn}

Since our families of partitions are directed under refinement, the function $\phi$ in Definition~\ref{dfn:adequate-conv} may be regarded as an extended-real-valued net indexed by a directed family.  Then Definition~\ref{dfn:adequate-conv} is simply an instance of convergence of a net in the sense of Moore and Smith: see, for instance,~\cite[pp28--30]{Dud--book}.  We assume some basic properties of limits of nets in the sequel; the properties we need are proved in the same way as for sequences.

In the sequel, we sometimes abbreviate $\lim_{\P \in \frak{P}}$ to $\lim_\P$ when the choice of directed family is clear from the context.

We may rewrite~\eqref{eq:H-gen-dfn} and~\eqref{eq:I-gen-dfn} as limits according to the following classical results of Dobrushin: see~\cite[Theorems 2.4.1 and 3.5.1]{Pin64}.

\begin{prop} \label{prop:Dob}
Suppose that $X$ and $Y$ take values in the measurable spaces $K$ and $L$ respectively, and let $\frak{P}$ and $\frak{Q}$ be any directed and generating families of partitions of $K$ and $L$.  In case $K$ is a finite set, we have
\begin{equation}\label{eq:H-flex}
	\rmH(X\,|\,Y) = \lim_{\Q \in \frak{Q}}\rmH(X\,|\,[Y]_\Q).
\end{equation}
For any choice of $K$ and $L$, we have
\begin{equation}\label{eq:mut-inf-flex}
\rmI(X\,;\,Y) = \lim_{\P \in \frak{P},\Q \in \frak{Q}}\rmI([X]_\P\,;\,[Y]_\Q).
\end{equation}
\end{prop}

To complete the generalization of basic information theory, we must define conditional mutual information in general.  This is done by combining~\eqref{eq:pre-cond-mut-inf} and~\eqref{eq:I-gen-dfn}: if $X$, $Y$ and $Z$ are general random variables, then we set
\begin{equation}\label{eq:gen-cond-mut-inf}
\rmI(X\,;\,Y\,|\,Z) := \sup_{\P,\Q}\rmI([X]_\P\,;\,[Y]_\Q\,|\,Z),
\end{equation}
where $\P$ and $\Q$ are as in~\eqref{eq:I-gen-dfn}.  For each fixed $\P$ and $\Q$ here, the right-hand side is defined by~\eqref{eq:pre-cond-mut-inf}, which in turn is a difference of quantities defined by infima as in~\eqref{eq:H-gen-dfn}.  In light of Proposition~\ref{prop:Dob}, we now recognize~\eqref{eq:gen-cond-mut-inf} as an iterated limit.  Suppose that $X$, $Y$ and $Z$ take values in $K$, $L$ and $M$ respectively, and let $\frak{P}$, $\frak{Q}$ and $\frak{R}$ be any directed and generating families for those respective spaces.  Then Proposition~\ref{prop:Dob} gives
\begin{equation}\label{eq:cond-mut-inf-flex}
	\rmI(X\,;\,Y\,|\,Z) = \lim_{\P \in \frak{P},\Q \in \frak{Q}}\lim_{\R \in \frak{R}}\rmI([X]_\P\,;\,[Y]_\Q\,|\,[Z]_\R).
	\end{equation}

In general, the order of the iterated limit in~\eqref{eq:cond-mut-inf-flex} is important, and one cannot write it simply as the joint limit over $\P$, $\Q$ and $\R$.

\begin{ex}\label{ex:bad-I}
	Let $K := L := (\bbZ/2\bbZ)^\bbN$ and let $M := K\times L$.  Let $X = (X_i)_{i\geq 1}$ and $Y = (Y_i)_{i\geq 1}$ take values in $K$ and $L$ respectively, and assume that all the $X_i$s and $Y_i$s are uniform and independent.  Finally, let $Z := (X,Y)$.
	
	Since $Z$ determines both $X$ and $Y$, the definition~\eqref{eq:gen-cond-mut-inf} gives that $\rmI(X\,;\,Y\,|\,Z) = 0$.  (Once we know $Z$, there is no information left for $X$ and $Y$ to share!)
	
	However, now let $\P_i$ and $\Q_i$ be the partitions of $K$ and $L$ generated by $(X_j)_{j=1}^{2i}$ and $(Y_j)_{j=1}^{2i}$, respectively, and let $\R_i$ be the partition of $M$ generated by the random variables $X_j$ for $1 \leq j \leq i$ and $X_j + Y_j (\!\!\!\mod 2)$ for $1 \leq j \leq 2i$.  Clearly the partitions $\P_i$ (respectively, $\Q_i$, $\R_i$) become finer as $i\to\infty$ and generate $\B_K$ (respectively, $\B_L$, $\B_M$). But for each fixed $i$ we have
	\begin{align*}
\rmI([X]_{\P_i}\,;\,[Y]_{\Q_i}\,|\,[Z]_{\R_i}) &= \sum_{j=1}^i \rmI(X_i\,;\,Y_i\,|\,X_i,X_i + Y_i) + \sum_{j=i+1}^{2i} \rmI(X_i\,;\,Y_i\,|\,X_i + Y_i)\\
&= 0 + i = i,
\end{align*}
by the independence among all the $X_i$s and $Y_i$s.  As $i\to\infty$ this tends to $\infty$, even though $\rmI(X\,;\,Y\,|\,Z) = 0$.
\fin
	\end{ex}

Mutual information still satisfies the chain rule for general random variables $X$, $Y$ and $Z$.  We give a quick proof below which illustrates the use of limits along families of partitions.  A proof using conditional distributions (and hence requiring standard Borel spaces) goes back to~\cite{Dob59}; see~\cite[equation $(3.6.3)$]{Pin64}, and also the translator's notes to~\cite[Chapter 3]{Pin64}.

\begin{lem}\label{lem:mut-inf-chain}
 Any random variables $X$, $Y$ and $Z$ satisfy
	\[\rmI(X\,;\,Y,Z) = \rmI(X\,;\,Z) + \rmI(X\,;\,Y\,|\,Z).\]
\end{lem}

\begin{Proof}
	This is a textbook result if $X$, $Y$ and $Z$ are all finite-valued, already quoted above.
	
	In the general case, let $X$, $Y$ and $Z$ take values in $K$, $L$ and $M$ respectively.  Let $\frak{P}$, $\frak{Q}$ and $\frak{R}$ be the families of all finite measurable partitions of those respective spaces (see Examples~\ref{exs:adequate}, Item 1).  Then two appeals to Proposition~\ref{prop:Dob} give
	\[\rmI(X\,;\,Z) = \lim_{\P,\R}\rmI([X]_\P\,;\,[Z]_\R)\]
	and
	\[\rmI(X\,;\,Y,Z) = \lim_{\P,\Q,\R}\rmI([X]_\P\,;\,[Y,Z]_{\Q \times \R}) = \lim_{\P,\Q,\R}\rmI([X]_\P\,;\,[Y]_\Q,[Z]_\R).\]
	The second of these formulae uses the construction from Examples~\ref{exs:adequate}, Item 2: in quantizing $(Y,Z)$, we may restrict attention to partitions of the form $\Q\times \R$, since these form a directed and generating family for $L\times M$.

Moreover, since mutual information is non-decreasing under refinement of quantizations, the joint limits above agree with any corresponding iterated limits. Therefore we also have
\[\rmI(X\,;\,Z) = \lim_\P\lim_\R\rmI([X]_\P\,;\,[Z]_\R) \quad \hbox{and} \quad \rmI(X\,;\,Y,Z) = \lim_\P\lim_\Q\lim_\R\rmI([X]_\P\,;\,[Y]_\Q,[Z]_\R).\]
	
	Finally, formula~\eqref{eq:cond-mut-inf-flex} gives
	\[\rmI(X\,;\,Y\,|\,Z) = \lim_{\P,\Q}\lim_\R\rmI([X]_\P\,;\,[Y]_\Q\,|\,[Z]_\R) = \lim_\P\lim_\Q\lim_\R\rmI([X]_\P\,;\,[Y]_\Q\,|\,[Z]_\R).\]
The order of limits is important in this equality, but not in the previous ones.  Once again, the second equality holds here because $\lim_\R\rmI([X]_\P\,;\,[Y]_\Q\,|\,[Z]_\R)$ is non-decreasing in $\P$ and $\Q$.

With these limit formulae in hand, we can apply the particular iterated limit $\lim_\P\lim_\Q\lim_\R$ to the finite-valued chain rule
	\[\rmI([X]_\P\,;\,[Y]_\Q,[Z]_\R) = \rmI([X]_\P\,;\,[Z]_\R) + \rmI([X]_\P\,;\,[Y]_\Q\,|\,[Z]_\R),\]
and we arrive at the chain rule in general.
\end{Proof}

\subsection{Kullback--Leibler divergence}

Let $\mu,\nu$ be two probability measures on the same measurable space $K$.  The Kullback--Leibler (`KL') divergence between $\nu$ and $\mu$ is defined to be $+\infty$ unless $\nu \ll\mu$, and in that case it is
\[\int\log\frac{\d\nu}{\d\mu}\,\d\nu\]
(which may still be $+\infty$).  This value is denoted by $\rmD(\nu\,\|\,\mu)$.  It is non-negative, and zero if and only if $\nu = \mu$.

KL divergence arises naturally in information theory and large deviations as a comparison between two distributions: see, for instance,~\cite[Chapters 2 and 11]{CovTho06} for the case of distributions on finite sets, and~\cite[Section 2.4]{Pin64} or~\cite[Appendix D.3]{DemZei--LDPbook} for the general case.

KL divergence is related to entropy and mutual information by various identities.  In particular, suppose that $X$ and $Y$ are arbitrary random variables, say taking values in $K$ and $L$, let $\l \in \Pr(K\times L)$ be the joint distribution of $X$ and $Y$, and let $\mu$ and $\nu$ be the marginals of $\l$.  Then
\begin{equation}\label{eq:I-and-D}
\rmI(X\,;\,Y) = \rmD(\l\,\|\,\mu\times \nu).
\end{equation}
A classical result of Gelfand, Kolmogorov, Yaglom and Perez equates this with the definition in~\eqref{eq:I-gen-dfn}: see~\cite[Theorem 2.4.2]{Pin64} and the translator's second note to that chapter of~\cite{Pin64}.  Equation~\eqref{eq:I-and-D} itself appears as~\cite[equation $(2.4.4)$]{Pin64}.

KL divergence plays a role in this paper through formulae for TC and DTC that generalize~\eqref{eq:I-and-D}: see Section~\ref{sec:ref-meas}.  The formula for TC obtained there is important during the proof of Theorem A.

Like entropy and mutual information, KL divergence satisfies a chain rule.  We need the following special case of this in the sequel.  Assume that $K$ and $L$ are standard Borel. Let $\l$ be a probability measure on $K\times L$, let $\mu$ and $\nu$ be its marginals on $K$ and $L$ respectively, and let $\nu_\bullet$ be a kernel from $K$ to $L$ such that $\l = \mu\ltimes \nu_\bullet$.  Also, let $\mu'$ and $\nu'$ be two other probability measures on $K$ and $L$ respectively.  Then
\begin{equation}\label{eq:KL-chain}
\rmD(\l\,\|\,\mu'\times \nu') = \rmD(\mu\,\|\,\mu') + \int \rmD(\nu_x\,\|\,\nu')\,\mu(\d x).
\end{equation}
A more general version of~\eqref{eq:KL-chain} decomposes $\rmD(\l\,\|\,\l')$ for any pair of probability measures $\l$, $\l'$ on $K\times L$, but we do not need this below.  For distributions on finite sets, the chain rule for KL divergence can be found in standard references such as~\cite[Theorem 2.5.3]{CovTho06}.  The proof of that special case already shows the essential calculation.  Dobrushin's paper~\cite{Dob59} seems to have been the first to treat the general case, which involves some extra analytic considerations.    That generality can also be found in~{\cite[equation~$(3.11.5)$]{Pin64}},  where it is attributed to Kolmogorov, and in~\cite[Theorem D.13]{DemZei--LDPbook}.

Later we make use of~\eqref{eq:KL-chain} through various special cases.  If we let $\mu' = \mu$ in~\eqref{eq:KL-chain}, then the first right-hand term vanishes, leaving
\begin{equation}\label{eq:KL-chain2}
	\rmD(\l\,\|\,\mu\times \nu') = \int \rmD(\nu_x\,\|\,\nu')\,\mu(\d x).
	\end{equation}
Now let us insert~\eqref{eq:KL-chain2} back into~\eqref{eq:KL-chain}, then do the same with the roles of the two coordinates reversed, and then apply~\eqref{eq:I-and-D}:
\begin{align}\label{eq:KL-chain3}
	\rmD(\l\,\|\,\mu'\times \nu') &= \rmD(\mu\,\|\,\mu') + \rmD(\l\,\|\,\mu\times \nu') \nonumber \\
	&= \rmD(\mu\,\|\,\mu') + \rmD(\nu\,\|\,\nu') +  \rmD(\l\,\|\,\mu\times \nu) \nonumber\\
	&= \rmD(\mu\,\|\,\mu') + \rmD(\nu\,\|\,\nu') +  \rmI(X\,;\,Y).
\end{align}
In case $\mu = \mu'$ and $\rmD(\nu\,\|\,\nu') < \infty$, we may re-arrange~\eqref{eq:KL-chain3} and make another substitution from~\eqref{eq:KL-chain2}:
\begin{equation}\label{eq:KL-chain4}
\rmI(X\,;\,Y) = \rmD(\l\,\|\,\mu\times \nu') - \rmD(\nu\,\|\,\nu') = \int \rmD(\nu_x\,\|\,\nu')\,\mu(\d x) - \rmD(\nu\,\|\,\nu').
\end{equation}
Finally, if we also let $\nu' = \nu$, then this simplifies to
\begin{equation}\label{eq:KL-chain5}
\rmI(X\,;\,Y) = \int \rmD(\nu_x\,\|\,\nu)\,\mu(\d x).
\end{equation}

\subsection{Mutual information and mixtures}\label{subs:mix-mut-inf}

If $P \ltimes \mu_\bullet$ is a probability measure on $\O\times K$, and if $Z,Y$ have joint distribution $P\ltimes \mu_\bullet$, then we define the \textbf{mutual information} in $P$ and $\mu_\bullet$ to be $\rmI(Z\,;\,Y)$.  Clearly this quantity depends only on the joint distribution $P \ltimes \mu_\bullet$.  We may also refer to the `mutual information in the mixture' $\int \mu_\bullet\,\d P$, and denote it by $\rmI(P,\mu_\bullet)$.  In the case of a finite mixture
\[\mu = p_1\mu_1 + \cdots + p_m\mu_m\]
this mutual information is bounded by $\rmH(p_1,\dots,p_m)$, and hence by $\log m$.
\medskip

\section{Total correlation and dual total correlation}

\subsection{Definitions}\label{subs:dfns}

Consider random variables $X_1$, \dots, $X_n$ defined on the same probability space and taking values in measurable spaces $K_1$, \dots, $K_n$.  Like mutual information, TC and DTC are defined by an explicit formula in case the range spaces $K_i$ are finite, and by a supremum over quantizations in the general case.

\begin{dfn}\label{dfn:TC}
 If $K_1$, \dots, $K_n$ are finite sets, then the \textbf{total correlation} (`\textbf{TC}') of $X_1$, \dots, $X_n$ is
\[\TC(X_1\,;\,\dots\,;\,X_n):= \Big(\sum_{i=1}^n\rmH(X_i)\Big) - \rmH(X_1,\dots,X_n).\]
In the general case, their total correlation is
\begin{equation}\label{eq:TC-gen-dfn}
\TC(X_1\,;\,\dots\,;\,X_n) := \sup_{\P_1,\dots,\P_n}\TC([X_1]_{\P_1}\,;\,\dots\,;\,[X_n]_{\P_n}),
\end{equation}
where this supremum runs over all tuples of finite measurable partitions $\P_i$ of the spaces $K_i$.  This supremum may equal $+\infty$.
\end{dfn}

For discrete random variables, total correlation was introduced by Watanabe in~\cite{Wat60}, and this is his terminology.  It goes by various other names, such as `multi-information' in~\cite{AbdPlu10,AyOlbBerJos06,SchnstiBerBia03}.

In the sequel we frequently write $[n] := \{1,2,\dots,n\}$ when $n\in\bbN$, and we generally abbreviate $[n]\setminus \{i\} =: [n]\setminus i$ when $i \in [n]$.  We set $[0] := \emptyset$ to allow for some degenerate cases. Given any tuple $x = (x_1,\dots,x_n)$ and nonempty $S\subseteq [n]$, we write $x_S := (x_i)_{i\in S}$.  We also write $X_S:= (X_i)_{i \in S}$ for the random variables, and we let $X_\emptyset$ be an arbitrary choice of deterministic (hence `trivial') random variable.  This is so that expressions such as  $\rmI(Y\,;\,X_\emptyset)$ still make sense: for example, this one simply equals zero.

\begin{dfn}\label{dfn:DTC}
 If $K_1$, \dots, $K_n$ are finite sets, then the \textbf{dual total correlation} (`\textbf{DTC}') of $X_1$, \dots, $X_n$ is
	\[\DTC(X_1\,;\,\dots\,;\,X_n):= \rmH(X_1,\dots,X_n) - \sum_{i=1}^n \rmH(X_i\,|\,X_{[n]\setminus i}).\]
In the general case, their dual total correlation is
\begin{equation}\label{eq:DTC-gen-dfn}
	\DTC(X_1,\dots,X_n) := \sup_{\P_1,\dots,\P_n}\DTC([X_1]_{\P_1}\,;\,\dots\,;\,[X_n]_{\P_n}),
\end{equation}
where this supremum runs over all tuples of finite measurable partitions $\P_i$ of the spaces $K_i$.  Once again, this may be $+\infty$ in general.
\end{dfn}

For discrete random variables, dual total correlation was first studied by Han in~\cite{Han75}.  It is sometimes called `excess entropy', as in~\cite{AyOlbBerJos06}, or `binding information', as in~\cite{AbdPlu10}.

The use of suprema to define TC and DTC in the general case is justified by the monotonicity properties proved in Subsection~\ref{subs:monotonicity} below.

It is clear that both TC and DTC depend only on the joint distribution $\mu$ of $X_1$, \dots, $X_n$, so we sometimes write $\TC(\mu)$ or $\DTC(\mu)$.  This leaves the relevant product structure to the reader's understanding.

\subsection{Basic identities}

The rest of this section is given to some basic identities and inequalities that relate TC, DTC, Shannon entropy, and mutual information.  They are used frequently during other proofs in later sections.

In the case of finite-valued random variables, these identities follow from easy applications of the chain rule.  However, the extension to the general case is tricky if we use the definitions as suprema in~\eqref{eq:TC-gen-dfn} and~\eqref{eq:DTC-gen-dfn}.  It is made much easier by the fact that these suprema can instead be understood as limits in the sense of Definition~\ref{dfn:adequate-conv}, which follows from the finite-valued case of the monotonicity properties in Lemmas~\ref{lem:TC-monotone} and~\ref{lem:DTC-monotone} below.  Those, in turn, are proved using the finite-valued cases of some of our other identities.  Therefore, although we formulate all of our results directly in their full generality, they are proved in two rounds: we first give all the proofs for finite-valued random variables, and then return to the general setting in Subsection~\ref{subs:gen-RVs}.

The next lemma provides some alternative formulae for TC and DTC.  These formulae do not require any limits or suprema, even for general random variables.

\begin{lem}\label{lem:alt-formulae}
 TC and DTC satisfy these identities:
	\begin{equation}\tag{\emph{a}}
\TC(X_1\,;\,\dots\,;\,X_n) = \sum_{i=1}^n\rmI(X_i\,;\,X_{[i-1]});
\end{equation}
\begin{equation}\tag{\emph{b}}
\DTC(X_1\,;\,\dots\,;\,X_n) = \sum_{i=1}^n\rmI(X_i\,;\,X_{[n]\setminus i}\,|\,X_{[i-1]}) = \sum_{i=1}^n\rmI(X_i\,;\,X_{\{i+1,\dots,n\}}\,|\,X_{[i-1]});
\end{equation}
and
\begin{equation}\tag{\emph{c}}
	\TC(X_1\,;\,\dots\,;\,X_n) + \DTC(X_1\,;\,\dots\,;\,X_n) = \sum_{i=1}^n \rmI(X_i\,;\,X_{[n]\setminus i}).
	\end{equation}
In particular, in each of these equations, one side equals $+\infty$ only if they both do.
	\end{lem}
\pagebreak

\begin{rmks} \makebox{} 
		\begin{enumerate}
		\item 	The right-hand sides in (a) and (b) depend on the order in which we label our random variables, as in the chain rule itself.  The symmetry of TC and DTC can be recovered by averaging over all possible orderings.  We use this trick later in the proof of Theorem A.
	\item Parts (a) and (c) appear as~\cite[equations~$(14)$ and~$(18)$]{TimAlfFleBeg14}.  In~\cite{AbdPlu10}, part (b) appears as equation (20) and part (c) appears inside the proof of Theorem 3.
	\end{enumerate}
\end{rmks}

\begin{Proof}[Finite-valued case.]
The middle and right-hand side in (b) are equal because, after conditioning on the variables $X_1$, \dots, $X_{i-1}$, they cannot make any contribution to the mutual information $\rmI(X_i\,;\,X_{[n]\setminus i}\,|\,X_{[i-1]})$.
	
	Using the chain rule to expand $\rmH(X_1,\dots,X_n)$ into $n$ terms, and inserting the result into Definitions~\ref{dfn:TC} and~\ref{dfn:DTC}, we have
	\[\TC(X_1\,;\,\dots\,;\,X_n) = \sum_{i=1}^n \big[\rmH(X_i) - \rmH(X_i\,|\,X_{[i-1]})\big]\]
	and
	\[\DTC(X_1,\dots,X_n) = \sum_{i=1}^n\big[\rmH(X_i\,|\,X_{[i-1]}) - \rmH(X_i\,|\,X_{[n]\setminus i})\big].\]
	Recalling formulae~\eqref{eq:mut-inf} and~\eqref{eq:pre-cond-mut-inf}, we arrive immediately at parts (a) and (b).
		
	On the other hand, if we add the definitions of TC and DTC for finite-valued random variables, then the term $\rmH(X_1,\dots,X_n)$ cancels to leave
	\[\TC(X_1\,;\,\dots\,;\,X_n) + \DTC(X_1\,;\,\dots\,;\,X_n) = \sum_{i=1}^n\big[\rmH(X_i) - \rmH(X_i\,|\,X_{[n]\setminus i})\big].\]
	Now another appeal to formula~\eqref{eq:mut-inf} gives part (c).	
	\end{Proof}

\bigskip

The next lemma gives recursive formulae for TC and DTC as the number of variables increases.

\begin{lem}\label{lem:recurse}
 If $n\geq 3$, then TC and DTC satisfy
	\begin{equation}\tag{\emph{a}}
		\TC(X_1\,;\,\dots\,;\,X_n) = \TC(X_1\,;\,\dots\,;\,X_{n-1}) + \rmI(X_n\,;\,X_{[n-1]})
	\end{equation}
	and
	\begin{equation}\tag{\emph{b}}
		\DTC(X_1\,;\,\dots\,;\,X_n) = \DTC(X_1\,;\,\dots\,;\,X_{n-1}) + \sum_{i=1}^{n-1} \rmI(X_i\,;\,X_n\,|\,X_{[n-1]\setminus i}).
	\end{equation}
	In particular, in both of these equations, one side equals $+\infty$ only if they both do.
\end{lem}
\pagebreak

\begin{rmk}
Part (b) appears as~\cite[equation $(13)$]{AyOlbBerJos06}.
\end{rmk}

\begin{Proof}[Finite-valued case.]
	To prove (a), let us substitute for $\TC(X_1\,;\,\dots\,;\,X_{n-1})$ from the right-hand side of Lemma~\ref{lem:alt-formulae}(a):
	\[ \TC(X_1\,;\,\dots\,;\,X_{n-1}) + \rmI(X_n\,;\,X_{[n-1]}) = \sum_{i=1}^{n-1}\rmI(X_i\,;\,X_{[i-1]}) + \rmI(X_n\,;\,X_{[n-1]}).\]
	This equals $\TC(X_1\,;\,\dots\,;\,X_n)$ by another use of Lemma~\ref{lem:alt-formulae}(a).
	
	Similarly, to prove (b), we substitute for $\DTC(X_1\,;\,\dots\,;\,X_{n-1})$ from the right-hand side of Lemma~\ref{lem:alt-formulae}(b):
	\begin{align*}
		&\DTC(X_1\,;\,\dots\,;\,X_{n-1}) + \sum_{i=1}^{n-1} \rmI(X_i\,;\,X_n\,|\,X_{[n-1]\setminus i}) \\&= \sum_{i=1}^{n-1}\rmI(X_i\,;\,X_{\{i+1,\dots,n-1\}}\,|\,X_{[i-1]}) + \sum_{i=1}^{n-1} \rmI(X_i\,;\,X_n\,|\,X_{[n-1]\setminus i})\\
		& = \sum_{i=1}^{n-1}\big(\rmI(X_i\,;\,X_{\{i+1,\dots,n-1\}}\,|\,X_{[i-1]}) +  \rmI(X_i\,;\,X_n\,|\,X_{[i-1]},X_{\{i+1,\dots,n-1\}})\big).
	\end{align*}
	By the chain rule for mutual information (Lemma~\ref{lem:mut-inf-chain}), this is equal to
	\[ \sum_{i=1}^{n-1}\rmI(X_i\,;\,X_{\{i+1,\dots,n\}}\,|\,X_{[i-1]}).\]
	Adding the dummy term $\rmI(X_n\,;\,X_\emptyset\,|\,X_{[n-1]}) = 0$, we see that this agrees with $\DTC(X_1\,;\newline \dots\,;\,X_n)$ by another use of Lemma~\ref{lem:alt-formulae}(b).
\end{Proof}

\begin{cor}\label{cor:recurse}
 If $n\geq 3$ then
	\[\TC(X_1\,;\,\dots\,;\,X_n) \geq \TC(X_1\,;\,\dots\,;\,X_{n-1})\]
	and
	\[\DTC(X_1\,;\,\dots\,;\,X_n) \geq \DTC(X_1\,;\,\dots\,;\,X_{n-1}).\]
\end{cor}

\smallskip

\subsection{Monotonicity properties}\label{subs:monotonicity}

This subsection introduces simple but useful comparisons of the TC or DTC values of two different $n$-tuples of random variables.

\begin{lem}\label{lem:TC-monotone}
 For any random variables $X_1$, \dots, $X_n$ and $Y_1$, \dots, $Y_n$ we have
	\[\TC(Y_1\,;\,\dots\,;\,Y_n) \leq \TC(X_1\,;\,\dots\,;\,X_n) + \sum_{i=1}^n\rmH(Y_i\,|\,X_i).\]
	In particular, if $X_i$ almost surely determines $Y_i$ for each $i$, then
	\begin{equation}\label{eq:TC-monotone}
	\TC(Y_1\,;\,\dots\,;\,Y_n) \leq \TC(X_1\,;\,\dots\,;\,X_n).
	\end{equation}
\end{lem}

\begin{Proof}[Finite-valued case.]
	By symmetry and induction it suffices to prove the first inequality when $Y_i = X_i$ for $i=1,2,\dots,n-1$.  In that case two appeals to Lemma~\ref{lem:recurse}(a) give
\begin{align*}
&\TC(X_1\,;\,\dots\,;\,X_{n-1}\,;\,Y_n) - \TC(X_1\,;\,\dots\,;\,X_{n-1}\,;\,X_n) \\ &= \rmI(Y_n\,;\,X_{[n-1]}) - \rmI(X_n\,;\,X_{[n-1]})\\
&\leq \rmI(Y_n,X_n\,;\,X_{[n-1]}) - \rmI(X_n\,;\,X_{[n-1]})\\
&= \rmI(Y_n\,;\,X_{[n-1]}\,|\,X_n) \qquad \qquad \qquad (\hbox{chain rule for}\ \rmI)\\
&\leq \rmH(Y_n\,|\,X_n).
\end{align*}
This proves the first inequality.  The second inequality is a special case.
\end{Proof}

\begin{lem}\label{lem:DTC-monotone}
 For any random variables $X_1$, \dots, $X_n$ and $Y_1$, \dots, $Y_n$ we have
	\[\DTC(Y_1\,;\,\dots\,;\,Y_n) \leq \DTC(X_1\,;\,\dots\,;\,X_n) + (n-1)\sum_{i=1}^n\rmH(Y_i\,|\,X_i).\]
	In particular, if $X_i$ almost surely determines $Y_i$ for each $i$, then
	\begin{equation}\label{eq:DTC-monotone}
\DTC(Y_1\,;\,\dots\,;\,Y_n) \leq \DTC(X_1\,;\,\dots\,;\,X_n).
\end{equation}
\end{lem}

\begin{Proof}[Finite-valued case.]
	As for TC, it suffices to prove the first inequality when $Y_i = X_i$ for $i=1,2,\dots,n-1$.  In that case two appeals to Lemma~\ref{lem:recurse}(b) give
	\begin{align*}
		&\DTC(X_1\,;\,\dots\,;\,X_{n-1}\,;\,Y_n) - \DTC(X_1\,;\,\dots\,;\,X_{n-1}\,;\,X_n)\\
		&= \sum_{i=1}^{n-1}\big[\rmI(X_i\,;\,Y_n\,|\,X_{[n-1]\setminus i}) - \rmI(X_i\,;\,X_n\,|\,X_{[n-1]\setminus i})\big].
		\end{align*}
	Each of these $n-1$ terms may be bounded as follows:
	\begin{align*}
	&\rmI(X_i\,;\,Y_n\,|\,X_{[n-1]\setminus i}) - \rmI(X_i\,;\,X_n\,|\,X_{[n-1]\setminus i}) \\
	&\leq \rmI(X_i\,;\,Y_n,X_n\,|\,X_{[n-1]\setminus i}) - \rmI(X_i\,;\,X_n\,|\,X_{[n-1]\setminus i})\\
	&= \rmI(X_i\,;\,Y_n\,|\,X_{[n]\setminus i}) \qquad \qquad \qquad (\hbox{chain rule for}\ \rmI)\\
	&\leq \rmH(Y_n\,|\,X_{[n]\setminus i})\\
	&\leq \rmH(Y_n\,|\,X_n) \qquad \qquad \qquad \hbox{(since $n \in [n]\setminus i$)}.
	\end{align*}
Therefore our original difference of DTC-values is at most $(n-1)\rmH(Y_n\,|\,X_n)$, as required for the first inequality. The second inequality is a special case.
\end{Proof}

\bigskip

\begin{rmk}
		I would not be surprised if~\eqref{eq:TC-monotone} and~\eqref{eq:DTC-monotone} were already in the literature somewhere, but I have not found a reference.
	\end{rmk}

The factor of $(n-1)$ in the statement of Lemma~\ref{lem:DTC-monotone} cannot be improved in general.  This is seen in the following example, which is essentially taken from~\cite[Example 3.6]{Aus--WP}

\begin{ex}\label{ex:near-prod-big-DTC}
In the product group $(\bbZ/p\bbZ)^n$, consider the subset
	\[Z := \big\{(a_1,\dots,a_n) \in (\bbZ/p\bbZ)^n:\ a_1 + \cdots + a_n = 0\big\}\]
	Let $\mu$ be the uniform distribution on $Z$ and let $Y_1$, \dots, $Y_n$ be random variables with joint distribution $\mu$. Let $X_i := Y_i$ for $i\leq n-1$, and let $X_n$ be a uniform random element of $\bbZ/p\bbZ$ independent of all the other random variables.
Then $\DTC(X_1\,;\,\dots\,;\,X_n) = 0$, since all the $X_i$s are independent, while a simple calculation gives  $\DTC(Y_1\,;\,\dots\,;\,Y_n) = (n-1)\log p$. Finally,
\[\rmH(Y_i\,|\,X_i) = \left\{\begin{array}{ll}0 &\quad \hbox{if}\ i\leq n-1\\ \log p&\quad \hbox{if}\ i=n.\end{array}\right.\]
So the inequality of Lemma~\ref{lem:DTC-monotone} becomes an equality in this example. \fin
\end{ex}

It may be possible to improve Lemma~\ref{lem:DTC-monotone} by using some other information theoretic quantities in the bound, in addition to the conditional entropies $\rmH(Y_i\,|\,X_i)$, but we do not explore that possibility here.

\subsection{Extension to general random variables}\label{subs:gen-RVs}

For random variables $X_1$, \dots, $X_n$ taking values in general measurable spaces $K_1$, \dots, $K_n$, the definitions of TC and DTC in~\eqref{eq:TC-gen-dfn} and~\eqref{eq:DTC-gen-dfn} are tractable using the monoticity provided by Lemmas~\ref{lem:TC-monotone} and~\ref{lem:DTC-monotone}.  The next result is the analog of Proposition~\ref{prop:Dob}.

\begin{lem}\label{lem:part-flex}
 Let $\frak{P}_i$ be a directed and generating collection of partitions of $K_i$ for each $i$ (recall Definition~\ref{dfn:adequate}).  Then
	\begin{equation}\label{eq:part-flex}
	\TC(X_1,\dots,X_n) = \lim_{\P_1 \in \frak{P}_1,\dots,\P_n \in \frak{P}_n}\TC([X_1]_{\P_1}\,;\,\dots\,;\,[X_n]_{\P_n})
	\end{equation}
	and similarly with DTC in place of TC.
	\end{lem}

\begin{Proof}
The inequality `$\geq$' in~\eqref{eq:part-flex} is immediate from the definition~\eqref{eq:TC-gen-dfn}.  So now let $\Q_i$ be an arbitrary finite measurable partition of $K_i$ for each $i$, and let $\eps > 0$.  Since each $\frak{P}_i$ is directed and generating, there are members $\P_i \in \frak{P}_i$ such that
	\[\rmH([X_i]_{\Q_i}\,|\,[X_i]_{\P_i}) < \eps/n \quad \hbox{for each}\ i \in [n].\]
	Now Lemma~\ref{lem:TC-monotone} gives
	\[\TC([X_1]_{\P'_1}\,;\,\dots\,;\,[X_n]_{\P'_n}) > \TC([X_1]_{\Q_1}\,;\,\dots\,;\,[X_n]_{\Q_n}) - \eps\]
	whenever $\P'_i \in \frak{P}_i$ is a further refinement of $\P_i$ for each $i$. Since $\eps > 0$ was arbitrary, this completes the proof of~\eqref{eq:part-flex}. The argument for DTC is exactly analogous.
	\end{Proof}

\bigskip

TC and DTC give two different ways to quantify the dependence among a tuple of random variables.  In general they can give very different values.  However, they are not completely independent; in particular, one cannot be infinite unless they both are.  This is a corollary of the following more precise inequalities.

\begin{lem}\label{lem:basic-ineqs}
 The TC and DTC of $X_1$, \dots, $X_n$ both lie between
	\begin{equation}\label{eq:max-TC-DTC-max}
	\max_i \rmI(X_i\,;\,X_{[n]\setminus i}) \quad \hbox{and} \quad (n-1)\cdot \max_i\rmI(X_i\,;\,X_{[n]\setminus i}).
	\end{equation}
	In particular,
	\[\DTC \leq (n-1)\cdot \TC \quad \hbox{and} \quad \TC \leq (n-1)\cdot \DTC,\]
	and the quantities TC, DTC and $\max_i \rmI(X_i\,;\,X_{[n]\setminus i})$ are either all finite or all infinite.
\end{lem}

\begin{Proof}
	\emph{Step 1.}\quad Suppose first that each $K_i$ is finite.  Then the finite-valued cases of parts (a) and (b) of Lemma~\ref{lem:alt-formulae} give
	\[\TC(X_1\,;\,\dots\,;\,X_n) \geq \rmI(X_n\,;\,X_{[n]\setminus n})\]
	and
	\[\DTC(X_1\,;\,\dots\,;\,X_n) \geq \rmI(X_1\,;\,X_{[n]\setminus 1})\]
	respectively, since all the terms in those sums are non-negative. This argument may be applied for any re-ordering of the random variables, so it implies the lower bound in~\eqref{eq:max-TC-DTC-max}.  Now we can insert that lower bound into the left-hand side of Lemma~\ref{lem:alt-formulae}(c) and then cancel a common term, giving the required upper bound.
	
	\vspace{7pt}
	
	\emph{Step 2.}\quad Now consider arbitrary measurable spaces $K_1$, \dots, $K_n$. For any finite measurable partitions $\P_1$, \dots, $\P_n$ of those spaces, Step 1 above gives
	\begin{multline}\label{eq:applying-step-1}
\max_i \rmI([X_i]_{\P_i}\,;\,[X_{[n]\setminus i}]_{\P_{[n]\setminus i}}) \\ \le \TC([X_1]_{\P_1}\,;\,\dots\,;\,[X_n]_{\P_n}),\ \  \DTC([X_1]_{\P_1}\,;\,\dots\,;\,[X_n]_{\P_n}) \\ \le (n-1)\cdot \max_i \rmI([X_i]_{\P_i}\,;\,[X_{[n]\setminus i}]_{\P_{[n]\setminus i}}),
\end{multline}
where $\P_{[n]\setminus i}$ denotes the partition of $\prod_{j \in [n]\setminus i}K_j$ into the sets $\prod_{j \in [n]\setminus i}C_j$ with $C_j \in \P_j$.  Taking the limit along $\P_1$, \dots, $\P_n$ in~\eqref{eq:applying-step-1} proves the bounds~\eqref{eq:max-TC-DTC-max} in general.

\vspace{7pt}

\emph{Step 3.}\quad The remaining assertions of Lemma~\ref{lem:basic-ineqs} follow directly from the bounds~\eqref{eq:max-TC-DTC-max}.\\
\makebox{}
\end{Proof}

\begin{dfn}\label{dfn:fin-corr}
 The random variables $X_1$, \dots, $X_n$ are \textbf{finitely correlated} if any (and hence all) of the quantities
	\[\TC(X_1\,;\,\dots\,;\,X_n), \quad \DTC(X_1\,;\,\dots\,;\,X_n), \quad \max_i\rmI(X_i\,;\,X_{[n]\setminus i})\]
	are finite.  We also refer to a joint distribution of such random variables as \textbf{finitely correlated}.  If this property does not hold, then the random variables and their joint distribution are \textbf{infinitely correlated}.
\end{dfn}

Clearly finite-valued random variables are always finitely correlated, and so are arbitrary independent random variables.

We can now extend our previous identities for TC and DTC to general random variables.  In the following three proofs, $\frak{P}_i$ is the family of all finite measurable partitions of $K_i$ for each $i$.

\begin{Proof}[Lemma~\ref{lem:alt-formulae} in general case.]\\[1mm]
	Parts (a) and (c) follow by applying the finite-valued case to the quantizations $[X_1]_{\P_1}$, \dots, $[X_n]_{\P_n}$ and then taking the limit along $\P_1 \in \frak{P}_1$, \dots, $\P_n \in \frak{P}_n$ on both sides.
	
	The proof of part (b) is complicated by the presence of conditional mutual information values on the right. Recall from Subsection~\ref{subs:basic-gen-RVs} that, if these are to be handled as limits, then they require iterated limits taken in the right order.  However, we can instead give an indirect proof of (b) using (a) and (c).
	
	If $X_1$, \dots, $X_n$ are infinitely correlated, then both sides of (b) are infinite.  If they are finitely correlated, then every quantity appearing in (a) or (c) is finite, and therefore so is every quantity appearing in (b) by the chain rule for mutual information (Lemma~\ref{lem:mut-inf-chain}).  As a result, we may subtract (c) from (c) to leave exactly (b), by another use of that chain rule.
	\end{Proof}

\begin{Proof}[Lemma~\ref{lem:recurse} and Corollary~\ref{cor:recurse} in general case.]\\[1mm]
	Now that we have the general case of Lemma~\ref{lem:alt-formulae}, these two results follow from it exactly as in the finite-valued case. All of the required re-arrangements are still correct if some of the values are equal to $+\infty$, because no cancellation or subtraction are involved.
\end{Proof}

\begin{Proof}[Lemmas~\ref{lem:TC-monotone} and~\ref{lem:DTC-monotone} in general case.]\\[1mm]
 Suppose that $Y_1$, \dots, $Y_n$ take values in $L_1$, \dots, $L_n$, and let $\frak{Q}_i$ be the family of all finite measurable partitions of $L_i$ for each $i$.  The finite-valued case of Lemma~\ref{lem:TC-monotone} gives
\[\TC([Y_1]_{\Q_1}\,;\,\dots\,;\,[Y_n]_{\Q_n}) \leq \TC([X_1]_{\P_1}\,;\,\dots\,;\,[X_n]_{\P_n}) + \sum_{i=1}^n\rmH([Y_i]_{\Q_i}\,|\,[X_i]_{\P_i})\]
for any $\P_i \in \frak{P}_i$ and $\Q_i \in \frak{Q}_i$.  Taking the limit along $\P_1 \in \frak{P}_1$, \dots, $\P_n \in \frak{P}_n$, and using the monotonicity of Shannon entropy, this gives
\begin{align*}
\TC([Y_1]_{\Q_1}\,;\,\dots\,;\,[Y_n]_{\Q_n}) &\leq \TC(X_1\,;\,\dots\,;\,X_n) + \sum_{i=1}^n\rmH([Y_i]_{\Q_i}\,|\,X_i)\\
&\le \TC(X_1\,;\,\dots\,;\,X_n) + \sum_{i=1}^n\rmH(Y_i\,|\,X_i).
\end{align*}
Now the limit along $\Q_1 \in \frak{Q}_1$, \dots, $\Q_n \in \frak{Q}_n$ completes the proof for TC.  The argument for DTC is analogous.
\end{Proof}

\section{TC and DTC in terms of reference measures}\label{sec:ref-meas}

Another approach to TC and DTC for general random variables starts with the formula~\eqref{eq:I-and-D} for mutual information in terms of KL divergence.  In the case of TC the relevant formula is already well-used in the literature:
\begin{equation}\label{eq:TC-and-D}
\TC(\mu) = \rmD(\mu\,\|\,\mu_1\times \cdots \times \mu_n),
\end{equation}
where $\mu_i$ is the $i^{\rm{th}}$ marginal of $\mu$.  
When $n=2$, this is just~\eqref{eq:I-and-D}.  The case $n\geq 3$ follows easily from that special case by induction, using Lemma~\ref{lem:recurse}(a) for the left-hand side and the general chain rule for KL divergence for the right-hand side.

In the form given by~\eqref{eq:TC-and-D}, TC has already played an important role in some results of theoretical probability, such as those of Csisz\'ar~\cite{Csi84} and Marton~\cite{Mar86,Mar96} mentioned in the Introduction.

If we assume that the measurable spaces $K_i$ are standard, then DTC also has a formula in terms of KL divergences.  Before we derive this, however, let us add another layer of generality.  Rather than compare $\mu$ to its own marginals, as in~\eqref{eq:TC-and-D}, one can compare it to a prior choice of reference measures on each $K_i$.

\begin{prop}\label{prop:TC-DTC-and-D}
 Assume that each $K_i$ is standard, and let $\l_i \in \Pr(K_i)$ be a reference measure for each $i$.
\begin{enumerate}
\item[(a)] If $\rmD(\mu_i\,\|\,\l_i) < \infty$ for each $i$, then $\TC(\mu)$ is equal to
\[\rmD(\mu\,\|\,\l_1\times \cdots \times \l_n) - \sum_{i=1}^n \rmD(\mu_i\,\|\,\l_i).\]
\item[(b)] If $\mu$ is finitely correlated and $\rmD(\mu_i\,\|\,\l_i) < \infty$ for each $i$, then $\DTC(\mu)$ is equal to
\begin{equation}\label{eq:DTC-and-D}
\sum_{i=1}^n \int \rmD(\mu_{i,z}\,\|\,\l_i)\,\mu_{[n]\setminus i}(\d z) - \rmD(\mu\,\|\,\l_1\times \cdots \times \l_n),
\end{equation}
where $\mu_{[n]\setminus i}$ is the projection of $\mu$ to $\prod_{j \in [n]\setminus i}K_j$, and $(\mu_{i,z}:\ z \in \prod_{j \in [n]\setminus i}K_j)$ is a conditional distribution for the $i^{\rm{th}}$ coordinate given the other coordinates according to $\mu$.
\end{enumerate}
\end{prop}

\begin{Proof}
\emph{Part (a).}\quad Given the formula~\eqref{eq:TC-and-D}, part (a) is equivalent to
\begin{equation}\label{eq:Csi}
	\rmD(\mu\,\|\,\l_1\times \cdots \times \l_n) = \rmD(\mu\,\|\,\mu_1\times \cdots \times \mu_n) + \sum_{i=1}^n \rmD(\mu_i\,\|\,\l_i).
	\end{equation}
This is a classical identity due to Csisz\'ar: see~\cite[equation $(2.11)$]{Csi84}.  It generalizes the earlier formula~\eqref{eq:KL-chain3} relating mutual information and KL divergences.  We may re-arrange Csisz\'ar's identity in the required way because we have assumed that each of the quantities $\rmD(\mu_i\,\|\,\l_i)$ is finite.

\vspace{7pt}

\emph{Part (b).}\quad This can be derived quickly by combining part (a) and Lemma~\ref{lem:alt-formulae}(c). Those ingredients give
\begin{equation}\label{eq:pre-DTC-and-D}
\DTC(\mu) = \sum_{i=1}^n \rmI(X_i\,;\,X_{[n]\setminus i}) - \Big[\rmD(\mu\,\|\,\l_1\times \cdots \times \l_n) - \sum_{i=1}^n \rmD(\mu_i\,\|\,\l_i)\Big],
\end{equation}
where $X_1$, \dots, $X_n$ are random variables with joint distribution $\mu$.  The bracketed expression on the right equals $\TC(\mu)$, and we may subtract it like this because we have assumed that $\mu$ is finitely correlated.

For each $i=1,2,\dots,n$, we now apply the special case~\eqref{eq:KL-chain4} of the chain rule for KL divergence:
\[\rmI(X_i\,;\,X_{[n]\setminus i}) = \int \rmD(\mu_{i,z}\,\|\,\l_i)\,\mu_{[n]\setminus i}(\d z) - \rmD(\mu_i\,\|\,\l_i).\]
Substituting this into the first sum in~\eqref{eq:pre-DTC-and-D}, the appearances of the quantities $\rmD(\mu_i\,\|\,\l_i)$ cancel, and we are left with~\eqref{eq:DTC-and-D}.
\end{Proof}

\begin{rmk}
In view of Csisz\'ar's identity~\eqref{eq:Csi}, the assumptions for Proposition~\ref{prop:TC-DTC-and-D}(b) are equivalent to the single assumption
\[\rmD(\mu\,\|\,\l_1\times \cdots \times \l_n) < \infty.\]
\end{rmk}

Since DTC is non-negative, Proposition~\ref{prop:TC-DTC-and-D}(b) implies that
\begin{equation}\label{eq:MadTet}
\rmD(\mu\,\|\,\l_1\times \cdots \times \l_n) \leq \sum_{i=1}^n \int \rmD(\mu_{i,z}\,\|\,\l_i)\,\mu_{[n]\setminus i}(\d z).
\end{equation}
This inequality is widely known.  It offers a simple explanation for the tensorization property of logarithmic Sobolev inequalities, as exposed in~\cite{Led9597}: see~\cite[Proposition 4.1]{Led9597} for an elementary family of inequalities that includes~\eqref{eq:MadTet}, attributed to Bobkov in that paper. As far as I know, it is a new observation that the gap in~\eqref{eq:MadTet} is actually the same quantity as DTC (which is the gap in the analogous Han inequality) after passing to a supremum over quantizations.

Before leaving this section, let us mention another setting that offers its own formulae for TC and DTC. If $X_1$, \dots, $X_n$ are real-valued and jointly absolutely continuous, then
\[\TC(\mu) = \sum_{i=1}^n \rmh(X_i) - \rmh(X_1,\dots,X_n)\]
and
\[\DTC(\mu) = \rmh(X_1,\dots,X_n) - \sum_{i=1}^n \rmh(X_i\,|\,X_{[n]\setminus i}).\]
Here $\rmh$ stands for Shannon's differential entropy (see~\cite[Chapter 8]{CovTho06}), and we assume enough smoothness of the joint PDF that the subtracted terms are all finite.  These identities generalize a standard formula for the mutual information between jointly continuous random variables: see~\cite[Section 8.5]{CovTho06}.  The proof of these formulae is analogous to the proof of Proposition~\ref{prop:TC-DTC-and-D}, and we omit it.  Alternatively, they may be subsumed into Proposition~\ref{prop:TC-DTC-and-D} itself if one expands the definition of KL divergence to allow arbitrary reference measures.

\section{Conditional TC and DTC}

\subsection{Definitions and first properties}

Consider a random $n$-tuple $X_1,\dots,X_n$ and another random variable $Y$ on the same probability space.  Assume that the $X_i$s are finite-valued, but let $Y$ be arbitrary. Then we define the \textbf{conditional total correlation} by
\begin{equation}\label{eq:cond-TC-dfn}
	\TC(X_1\,;\,\dots\,;\,X_n\,|\,Y) := \Big(\sum_{i=1}^n\rmH(X_i\,|\,Y)\Big) - \rmH(X_1,\dots,X_n\,|\,Y)
\end{equation}
and the \textbf{conditional dual total correlation} by
\begin{equation}\label{eq:cond-DTC-dfn}
	\DTC(X_1\,;\,\dots\,;\,X_n\,|\,Y) := \rmH(X_1,\dots,X_n\,|\,Y) - \sum_{i=1}^n\rmH(X_i\,|\,X_{[n]\setminus i},Y).
\end{equation}

In case the $X_i$s are also general random variables, we extend the definitions of conditional TC and DTC just as we did in the unconditional case:
\[\TC(X_1\,;\,\dots\,;\,X_n\,|\,Y) := \sup_{\P_1,\dots,\P_n}\TC\big([X_1]_{\P_1}\,;\,\dots\,;\,[X_n]_{\P_n}\,\big|\,Y\big)\]
and similarly with DTC in place of TC.  When $n=2$, this generalizes the definition of conditional mutual information in~\eqref{eq:gen-cond-mut-inf}.  Just as in Lemma~\ref{lem:part-flex} for unconditional TC and DTC, we may replace the supremum above with a limit along any directed and generating families of partitions $\frak{P}_i$. The proof is unchanged from the unconditional case.

\subsection{Expression in terms of disintegrations}

Let $X_1$, \dots, $X_n$ and $Y$ be random variables with respective target spaces $K_1$, \dots, $K_n$ and $L$, let $\mu$ be the joint distribution of $X_1$, \dots, $X_n$ on $\prod_i K_i$, and let $\nu$ be the distribution of $Y$ on $L$.  In case the spaces $K_i$ are standard Borel, $\mu$ has a disintegration $(\mu_y:\ y \in L)$ over $Y$.  Then conditional TC and DTC satisfy the obvious formula in terms of this disintegration, generalizing a classical identity for mutual information.

\begin{prop}\label{prop:TC-DTC-disint}
 In the situation above, we have
	\[\TC(X_1\,;\,\dots\,;\,X_n\,|\,Y) = \int \TC(\mu_y)\,\nu(\d y)\]
	and similarly with DTC in place of TC.
	\end{prop}

\begin{Proof}
	We give the proof for TC; the proof for DTC is analogous.
	
	Suppose first that each $K_i$ is finite.  In that case we have
	\[\TC(X_1\,;\,\dots\,;\,X_n\,|\,Y) = \sum_i \rmH(X_i\,|\,Y) - \rmH(X_1,\dots,X_n\,|\,Y).\]
	Let $(\mu_{i,y}:\ y\in L)$ be a conditional distribution for $X_i$ given $Y$.  Then $\mu_{i,y}$ is the $i^{\rm{th}}$ marginal of $\mu_y$ for almost every $y$, by the essential uniqueness of conditional distributions. Therefore, by the classical formula for conditional entropy in terms of conditional distributions, the above is equal to
	\[\int \Big[\sum_i\rmH(\mu_{i,y}) - \rmH(\mu_y)\Big]\,\nu(\d y) = \int \TC(\mu_y)\,\nu(\d y).\]
	
	For the general case, since each $K_i$ is standard Borel, it has a refining and generating sequence of finite partitions $\P_{i,1}$, $\P_{i,2}$, \dots.  For each finite $t$, the special case proved above gives
	\begin{equation}\label{eq:pre-TC-disint}
	\TC([X_1]_{\P_{1,t}}\,;\,\dots\,;\,[X_n]_{\P_{n,t}}\,|\,Y) = \int \TC([\mu_y]_{\P_{1,t}\times \cdots \times \P_{n,t}})\,\nu(\d y),
\end{equation}
where $[\mu_y]_{\P_{1,t}\times \cdots \times \P_{n,t}}$ is the discrete probability distribution induced on the product partition ${\P_{1,t}\times \cdots \times \P_{n,t}}$ by the measure $\mu_y$.  As $t\to\infty$:
	\begin{itemize}
		\item the left-hand side of~\eqref{eq:pre-TC-disint} converges to $\TC(X_1\,;\,\dots\,;\,X_n\,|\,Y)$, by the generalization of Lemma~\ref{lem:part-flex} to conditional random variables; and
		\item the right-hand side of~\eqref{eq:pre-TC-disint} converges to the integral of $\TC(\mu_y)$ with respect to $\nu(\d y)$, by Lemma~\ref{lem:part-flex} and the monotone convergence theorem.
		\end{itemize}
	\end{Proof}

\subsection{Clumping rules}\label{subs:clumping}

Let $X_1$, \dots, $X_n$ be random variables taking values in arbitrary measurable spaces $K_1$, \dots, $K_n$.  Let $[n] = S_1 \cup \cdots \cup S_m$ be a partition into non-empty subsets, and enumerate $S_j = {\{k_{j,1} < \dots < k_{j,s_j}\}}$ for each $j$.  Then the product spaces
\[\prod_{i=1}^n K_i \quad \hbox{and} \quad \prod_{j=1}^m\Big(\prod_{i\in S_j}K_i\Big)\]
have a canonical identification.  Corresponding to this, we may define the new $m$-tuple of random variable $Y_j:= X_{S_j}$ taking values in the product spaces $K_{S_j} := \prod_{i\in S_j}K_i$.  We refer to this construction as \textbf{clumping}, and to $Y_1$, \dots, $Y_m$ as \textbf{clumped} random variables.

In the statements of the next two lemmas, we need the following convention: for a single random variable $Z$ (that is, a `$1$-tuple') we always have
\[\TC(Z) = \DTC(Z) = 0.\]

\begin{lem}[TC clumping rules]\label{lem:TC-clumping}
 The random variables above satisfy
	\[\TC(X_1\,;\,,\dots\,;\,X_n) = \TC(Y_1\,;\,\dots\,;\,Y_m) + \sum_{j=1}^m \TC(X_{k_{j,1}}\,;\,\dots\,;\,X_{k_{j,s_j}}).\]
	In particular, one side is $+\infty$ only if they both are.
\end{lem}

\begin{rmk}
	The case $m=2$ of this rule appears as~\cite[equation~(4.1)]{Per77}, and~\cite[equation (5)]{AyOlbBerJos06} follows by applying a special case of this rule repeatedly.
\end{rmk}

\begin{Proof}
		By induction, we may assume that all but one of the sets $S_j$ are singletons --- applying that case to $m$ separate `clumpings' gives the general formula.  Having done so, we may also assume that our partition of $[n]$ is
	\[[n] = \{1,2,\dots,\ell\} \cup \{\ell+1\}\cup \cdots \cup \{n\}\]
	for some $\ell \in \{1,\dots,n-1\}$, by symmetry among the random variables.  In this case the desired formula is
	\begin{equation}\label{eq:TC-clumping-special}
	\TC(X_1\,;\,\dots\,;\,X_n) = \TC(X_1\,;\,\dots\,;\,X_\ell) + \TC\big(Y\,;\,X_{\ell+1}\,;\,\dots\,;\,X_n\big),
	\end{equation}
	where $Y := (X_1,\dots,X_\ell)$. This formula can be deduced quickly from part (a) of Lemma~\ref{lem:alt-formulae}:
	\begin{multline*}
\TC(X_1\,;\,\dots\,;\,X_n) = \sum_{i=1}^n \rmI(X_i\,;\,X_{[i-1]}) = \sum_{i=1}^\ell\rmI(X_i\,;\,X_{[i-1]}) + \sum_{i=\ell + 1}^n   \rmI(X_i\,;\,X_{[i-1]})\\
= \sum_{i=1}^\ell\rmI(X_i\,;\,X_{[i-1]}) + \sum_{i=\ell + 1}^n   \rmI\big(X_i\,;\,Y,X_{\ell+1},\dots,X_{i-1}\big).
\end{multline*}
By another appeal to Lemma~\ref{lem:alt-formulae}(a), the two sums on the line above are equal to the two right-hand terms in~\eqref{eq:TC-clumping-special}, respectively.  For the second sum, this becomes apparent after we insert the dummy term $\rmI(Y\,;\,X_\emptyset) = 0$.
\end{Proof}
\smallskip

\begin{lem}[DTC clumping rules]\label{lem:DTC-clumping}
 The random variables above satisfy
	\[		\DTC(X_1\,;\,,\dots\,;\,X_n) = \DTC(Y_1\,;\,\dots\,;\,Y_m) + \sum_{j=1}^m \DTC(X_{k_{j,1}}\,;\,\dots\,;\,X_{k_{j,s_j}}\,|\,X_{[n]\setminus S_j}).\]
	In particular, one side is $+\infty$ only if they both are.
\end{lem}
\smallskip

\begin{rmk}
A special case of this rule may be applied repeatedly to recover~\cite[equation $(14)$]{AyOlbBerJos06}, where DTC is called `excess entropy'.
\end{rmk}

\begin{Proof}
	Similarly to the proof of Lemma~\ref{lem:TC-clumping}, we may reduce to the case in which all but one of the sets $S_j$ are singletons, and then, by symmetry, to the case of the partition
	\[[n] = \{1\} \cup \cdots \cup \{\ell-1\} \cup \{\ell,\ell+1,\dots,n\}\]
	for some $\ell \in \{1,\dots,n-1\}$.  Note that this time we put the non-singleton last. For this partition the desired formula is
	\begin{equation}\label{eq:DTC-clumping-special}
	\DTC(X_1\,;\,\dots\,;\,X_n) = \DTC\big(X_1\,;\,\dots\,;\,X_{\ell-1}\,;\,Y\big) + \DTC(X_\ell\,;\,\dots\,;\,X_n\,|\,X_{[\ell-1]}),
	\end{equation}
	where $Y := (X_\ell,\dots,X_n)$. Now we proceed using the expression given by Lemma~\ref{lem:alt-formulae}(b): 
	\pagebreak
	
	\begin{align*}
	&\DTC(X_1\,;\,\dots\,;\,X_n)\\
	&= \sum_{i=1}^n \rmI(X_i\,;\,X_{\{i+1,\dots,n\}}\,|\,X_{[i-1]})\\
	&= \sum_{i=1}^{\ell-1} \rmI(X_i\,;\,X_{\{i+1,\dots,\ell-1\}},Y\,|\,X_{[i-1]}) + \sum_{i=\ell}^n\rmI(X_i\,;\,X_{\{i+1,\dots,n\}}\,|\,X_{[\ell-1]},X_\ell,\dots,X_{i-1}).
	\end{align*}
	By two more appeals to Lemma~\ref{lem:alt-formulae}(b), the two sums on the previous line are equal to the two terms on the right-hand side of~\eqref{eq:DTC-clumping-special}, respectively.  For the first sum, this becomes apparent after we insert the dummy term $\rmI(Y\,;\,X_\emptyset\,|\,X_{[\ell-1]}) = 0$.
\end{Proof}

\bigskip

The following special case of Lemma~\ref{lem:DTC-clumping} is quite intuitive, and is needed by itself in the sequel.

\begin{cor}\label{cor:mini-clumping}
 Any random variables $X_1$, \dots, $X_n$ and $Y$ satisfy
	\[\DTC(X_1\,;\,\dots\,;\,X_n\,;\,Y) = \rmI(X_1,\dots,X_n\,;\,Y) + \DTC(X_1\,;\,\dots\,;\,X_n\,|\,Y).\]
\end{cor}

\begin{cor}\label{cor:clumping}
 In the setting of Lemmas~\ref{lem:TC-clumping} and~\ref{lem:DTC-clumping}, we have
	\[\TC(Y_1\,;\,\dots\,;\,Y_m) \leq \TC(X_1\,;\,\dots\,;\,X_n),\]
	and similarly with DTC in place of TC. In particular, if $S, S^\rm{c}$ is a binary partition of $[n]$, then
	\[\rmI(X_S\,;\,X_{S^{\rm{c}}}) \leq \min\big\{\TC(X_1\,;\,\dots\,;\,X_n), \DTC(X_1\,;\,\dots\,;\,X_n)\big\}.\]
\end{cor}

The second inequality here generalizes the lower bound in~\eqref{eq:max-TC-DTC-max}.

\begin{Proof}
	The first pair of inequalities follows from Lemmas~\ref{lem:TC-clumping} and~\ref{lem:DTC-clumping} because TC and DTC are non-negative.  The second pair is a special case of the first, because both TC and DTC reduce to mutual information in the case of two random variables.
\end{Proof}

\bigskip

\noindent{\bf\large Part II: Structure of measures with low correlation}


\section{TC and product measures}\label{sec:TC-and-prod}

The following concentration inequality lies behind much of Part II of this paper.

\begin{prop}[Marton's transportation inequality for product measures]\label{prop:Marton}
 Let $(K_1,d_{K_1})$, \dots, $(K_n,d_{K_n})$ be complete and separable metric spaces of diameter at most $1$. If $\nu \in \Pr(\prod_i K_i)$ is a product measure, and $\mu \in \Pr(\prod_i K_i)$ is arbitrary, then
	\[\ol{d_n}(\mu,\nu) \leq \sqrt{\frac{1}{2n}\rmD(\mu\,\|\,\nu)}.\]
\end{prop}

In Marton's original proofs in~\cite{Mar86,Mar96}, each $K_i$ is equal to a fixed finite set and given its disrete metric, but those proofs generalize with only cosmetic changes to give the result stated above.

Combining Proposition~\ref{prop:Marton} with the identity~\eqref{eq:TC-and-D} gives the main structural result for measures with small TC.

\begin{cor}\label{cor:Marton}
 Any $\mu \in \Pr(\prod_i K_i)$ satisfies
	\[\ol{d_n}(\mu,\mu_1\times \cdots \times \mu_n) \leq \sqrt{\frac{1}{2n}\TC(\mu)},\]
	where $\mu_i$ is the $i^{\rm{th}}$ marginal of $\mu$.
\end{cor}

For finite alphabets with their discrete metrics, the relation between $\ol{d_n}$ and $\TC$ in Corollary~\ref{cor:Marton} can be partially reversed using the next lemma, which gives a kind of `continuity' for TC in the transportation metric.  It strengthens~\cite[Lemma 4.4]{Aus--WP}, but here we give a slightly more efficient proof based on Lemma~\ref{lem:TC-monotone}.

\begin{lem}\label{lem:TC-robust}
 Let $\mu,\nu \in \Pr(\prod_i K_i)$, and suppose that each $K_i$ has cardinality at most $k$ and is given its discrete metric: $d_{K_i}(x,y) := 1_{\{x\ne y\}}$.  Let $\delta := \ol{d_n}(\mu,\nu)$. Then
	\[|\TC(\mu) - \TC(\nu)| \leq \big(\rmH(\delta,1-\delta) + \delta \log(k-1)\big)n.\]
\end{lem}

\begin{Proof}
	Let $\l$ be a coupling of $\mu$ and $\nu$ that witnesses the distance $\ol{d_n}(\mu,\nu)$, and let $(X_1,Y_1)$, \dots, $(X_n,Y_n)$ be random pairs that take values in $K_1\times K_1$, \dots, $K_n\times K_n$ and have joint distribution $\l$.  Also, let
	\[\delta_i := \l\{(x,y):\ x_i\neq y_i\} \quad \hbox{for each}\ i=1,2,\dots,n.\]
	Then Lemma~\ref{lem:TC-monotone} and Fano's inequality give
\[\TC(\nu) \le \TC(\mu) + \sum_{i=1}^n \rmH(Y_i\,|\,X_i) \le \TC(\mu) + \sum_{i=1}^n\big(\rmH(\delta_i,1-\delta_i) + \delta_i \log(k-1)\big).\]
	Since we chose $\l$ to be an optimal coupling, $\delta$ is the average of the $\delta_i$s.  Therefore the concavity of the entropy function turns this into
	\[\TC(\nu) \leq \TC(\mu) + \big(\rmH(\delta,1-\delta) +\delta \log(k-1)\big)n.\]
	By symmetry in $\mu$ and $\nu$, this completes the proof.
\end{Proof}

\begin{cor}\label{cor:TC-robust}
 In the setting of Lemma~\ref{lem:TC-robust}, if $\nu$ is a product measure, then
\[\TC(\mu) \leq \big(\rmH(\delta,1-\delta) + \delta \log (k-1)\big)n.\]
	\end{cor}

Unlike the inequality in Corollary~\ref{cor:Marton}, the inequality in Corollary~\ref{cor:TC-robust} deteriorates for larger alphabets.  The next example shows that this feature of Corollary~\ref{cor:TC-robust} is essential.  For large or infinite alphabets, I do not know a refined description of exactly which near-product measures have small TC.

\begin{ex}\label{ex:near-prod-big-TC}
	Let $\delta > 0$ and $k \geq 2$, and let $K_i := \{0,1,\dots,k-1\}$ for each $i = 1,2,\dots,n$.  For each $j = 0,1,\dots,k-1$, let $\nu_j$ be the Dirac point mass on the $n$-tuple $(j,j,\dots,j)$.  Finally, let
	\[\mu := (1-\delta)\nu_0 + \frac{\delta}{k-1}\sum_{j=1}^{k-1}\nu_j.\]
	The product $\mu\times \nu_0$ is the only coupling of $\mu$ and $\nu_0$, and it gives
	\[\ol{d_n}(\mu,\nu_0) = \delta.\]
	However,
	\begin{align*}
\TC(\mu) &= n\cdot \rmH\Big(1-\delta,\frac{\delta}{k-1},\dots,\frac{\delta}{k-1}\Big) - \rmH\Big(1-\delta,\frac{\delta}{k-1},\dots,\frac{\delta}{k-1}\Big) \\
&\geq (n-1)\delta \log (k-1).
\end{align*}
If $k$ is large enough, then this TC can be a large multiple of $n$, even if $\delta$ is very small. \fin
	\end{ex}

In addition to giving Corollary~\ref{cor:Marton}, Proposition~\ref{prop:Marton} is ultimately responsible for the approximation between $\mu_y$ and a product measure $\xi_y$ in part (b) of Theorem A.

\section{DTC and mixtures of product measures}

\subsection{The DTC of a mixture of products}

\begin{prop}\label{prop:DTC-approx-concav}
 Let $\mu \in \Pr(\prod_i K_i)$, and let
	\begin{equation}\label{eq:another-mix}
	\mu = \int_L \mu_y\,\nu(\d y)
	\end{equation}
	be a representation of $\mu$ as a mixture.  Then
	\begin{equation}\label{eq:DTC-approx-concav}
	\DTC(\mu) \leq \int_L \DTC(\mu_y)\,\nu(\d y) + \rmI(\nu, \mu_\bullet).
	\end{equation}
	In particular, if every $\mu_y$ is a product measure, then $\DTC(\mu) \leq \rmI(\nu,\mu_\bullet)$.
	\end{prop}

\begin{Proof}
	Let $X_1$, \dots, $X_n$ be random variables with joint distribution $\mu$, and let $Y$ be another random variable so that the pair $Y$, $(X_1,\dots,X_n)$ is a randomization of the mixture~\eqref{eq:another-mix}.  In terms of these random variables, the right-hand side of~\eqref{eq:DTC-approx-concav} is
	\[\DTC(X_1\,;\,\dots\,;\,X_n\,|\,Y) + \rmI\big((X_1,\dots,X_n)\,;\,Y\big).\]
	This is equal to $\DTC(X_1\,;\,\dots\,;\,X_n\,;\,Y)$ by Corollary~\ref{cor:mini-clumping}, and this is greater than or equal to $\DTC(X_1\,;\,\dots\,;\,X_n)$ by Lemma~\ref{lem:recurse}(b).
	
	The final assertion follows because DTC is zero for any product measure.
	\end{Proof}

\bigskip

DTC is even more sensitive to small perturbations in $\ol{d_n}$ than TC: near-products can have large DTC even for small alphabets.  Because of this, a gap in our understanding remains between Proposition~\ref{prop:DTC-approx-concav} and Theorem A.

\begin{ex}\label{ex:near-prod-big-DTC2}
Let the measure $\mu \in \Pr((\bbZ/p\bbZ)^n)$ and random variables $X_i$ and $Y_i$ be as in Example~\ref{ex:near-prod-big-DTC}.  Let $\nu$ be the uniform distribution on $(\bbZ/p\bbZ)^n$, so this is the joint distribution of $X_1$, \dots, $X_n$. Then $\DTC(\mu) = (n-1)\log p$. But the construction of the $Y_i$s and $X_i$s gives a coupling of $\mu$ and $\nu$ under which (i) the first $n-1$ coordinates always agree and (ii) the last coordinates agree with probability $1/p > 0$, so $\ol{d_n}(\mu,\nu) < 1/n$. \fin
\end{ex}

\subsection{Proof of Theorem A}

Let $(X_1,\dots,X_n)$ be a random $n$-tuple with joint distribution $\mu$.

\begin{lem}\label{lem:DTC-ave}
 There is a subset $S \subseteq [n]$ with $|S| \leq \DTC(\mu)/\delta^2$ such that
\[	\TC(X_1\,;\,\dots\,;\,X_n\,|\,X_S) + \DTC(X_1\,;\,\dots\,;\,X_n\,|\,X_S) \leq \delta^2|[n]\setminus S|.\]
	\end{lem}
\smallskip

(In fact we need only the weaker conclusion ${\TC(X_1\,;\,\dots\,;\,X_n\,|\,X_S) \leq \delta^2|[n]\setminus S|}$ to prove Theorem A.)

\begin{Proof}
The formula from Lemma~\ref{lem:alt-formulae}(b) is available for any ordering of the index set $[n]$.  We may therefore consider its average over all possible orderings:
\begin{equation}\label{eq:DTC-ave}
\DTC(X_1\,;\,\dots\,;\,X_n) = \sum_{i=1}^n\bbE\big[\rmI\big(X_{\s(i)}\,;\,X_{[n]\setminus \s(i)}\,\big|\,X_{\{\s(1),\dots,\s(i-1)\}}\big)\big],
\end{equation}
where $\bbE = \frac{1}{n!}\sum_{\s}$ denotes expectation over a uniformly random permutation $\s$ of $[n]$.

	By Markov's inequality, fewer than $\DTC(\mu)/\delta^2$ of the summands in~\eqref{eq:DTC-ave} can exceed $\delta^2$, so some $i \leq \DTC(\mu)/\delta^2 + 1$ satisfies
\[\bbE\big[\rmI\big(X_{\s(i)}\,;\,X_{[n]\setminus \s(i)}\,\big|\,X_{\{\s(1),\dots,\s(i-1)\}}\big)\big] \leq \delta^2.\]
By the tower property of iterated conditional expectations, it follows that some subset $S \subseteq [n]$ of cardinality $i-1$ satisfies
\[\bbE\Big[\rmI\big(X_{\s(i)}\,;\,X_{[n]\setminus \s(i)}\,\big|\,X_{\{\s(1),\dots,\s(i-1)\}}\big)\,\Big|\,\{\s(1),\dots,\s(i-1)\} = S\Big] \leq \delta^2.\]
However, conditionally on $\{\s(1),\dots,\s(i-1)\} = S$, the random image $\s(i)$ is equally likely to be any element of $[n]\setminus S$.  Therefore the conditional expectation above is equal to
\[\frac{1}{|[n]\setminus S|}\sum_{j \in [n]\setminus S}\rmI(X_j\,;\,X_{[n]\setminus j}\,|\,X_S).\]
(We used a random permutation $\s$ in~\eqref{eq:DTC-ave} in order to arrive at this average over $j \in [n]\setminus S$.)  By Lemma~\ref{lem:alt-formulae}(c), this average is equal to
\[\frac{1}{|[n]\setminus S|}\big(\TC(X_1\,;\,\dots\,;\,X_n\,|\,X_S) + \DTC(X_1\,;\,\dots\,;\,X_n\,|\,X_S)\big).\]
\end{Proof}
\pagebreak

\begin{Proof}[Theorem A.]\\[1mm]
Let $S$ be given by Lemma~\ref{lem:DTC-ave}, so the hypotheses of Theorem A give $|S| \leq \DTC(\mu)/\delta^2 \leq \delta n$.  Let $\mu_S$ be the joint distribution of $X_S$, and similarly for any other subfamily of the random variables $X_i$.  Let $L := \prod_{i \in S}K_i$ and let ${(\nu_y:\ y \in L)}$ be a conditional distribution for $(X_1,\dots,X_n)$ given $X_S$.  Then $\mu$ is equal to the mixture
	\begin{equation}\label{eq:canon-mix}
	\mu = \int_L\nu_y\ \mu_S(\d y).
\end{equation}
At this stage we have no control on the mutual information of this mixture.  But we regain such control if we project to coordinates in $S^{\rm{c}} := [n]\setminus S$:
\begin{equation}\label{eq:proj-mix}
\mu_{S^{\rm{c}}} = \int_L (\nu_y)_{S^{\rm{c}}}\ \mu_S(\d y).
\end{equation}
Indeed, the mutual information in the projected mixture~\eqref{eq:proj-mix} is precisely the mutual information between $X_S$ and $X_{S^{\rm{c}}}$.  By the last inequality of Corollary~\ref{cor:clumping}, this is at most $\DTC(\mu)$.  In particular, it is finite, and so $(\nu_y)_{S^{\rm{c}}} \ll \mu_{S^{\rm{c}}}$ for $\mu_S$-almost every $y$ because of~\eqref{eq:KL-chain5}.

Let $d_{S^{\rm{c}}}$ be the normalized Hamming average of the metrics $d_{K_i}$ for $i \in S^{\rm{c}}$: that is, the analog of~\eqref{eq:Hamming} in which the sum extends only over $i \in S^{\rm{c}}$ and with normalizing constant $|S^{\rm{c}}|$.  Let $\ol{d_{S^{\rm{c}}}}$ be the transportation metric associated to $d_{S^{\rm{c}}}$, defined as in~\eqref{eq:trans-met} with $d_{S^{\rm{c}}}$ in place of $d_n$.

For each $y \in L$, let $\xi'_y$ be the product measure on $\prod_{i\in S^{\rm{c}}}K_i$ with the same marginals as $(\nu_y)_{S^{\rm{c}}}$.  These measures satisfy the following estimates:
\begin{align*}
	\int_L\ol{d_{S^{\rm{c}}}}\big((\nu_y)_{S^{\rm{c}}},\xi'_y)\ \mu_S(\d y) &\leq \int_L\sqrt{\frac{1}{2n}\TC\big((\nu_y)_{S^{\rm{c}}}\big)}\ \mu_S(\d y) \quad \hbox{(by Corollary~\ref{cor:Marton})}\\
	&\leq \sqrt{\frac{1}{2n}\int_L\TC\big((\nu_y)_{S^{\rm{c}}}\big)\ \mu_S(\d y)}  \quad \hbox{(by H\"{o}lder's ineq.)}\\
	&= \sqrt{\frac{1}{2n}\TC(X_1\,;\,\dots\,;\,X_n\,|\,X_S)}  \quad \hbox{(by Proposition~\ref{prop:TC-DTC-disint})}\\
	&\leq \delta/\sqrt{2} < \delta  \qquad \qquad \hbox{(by Lemma~\ref{lem:DTC-ave})}.
\end{align*}

To finish the construction: let $\xi_y$ be any lift of $\xi'_y$ to a product measure on $\prod_i K_i$, chosen measurably in $y$; let $\rho_y$ be the Radon--Nikodym derivative $\d(\nu_y)_{S^{\rm{c}}}/\d\mu_{S^{\rm{c}}}$; and define a new measure $\mu_y$ on $\prod_i K_i$ by
\[\mu_y(\d x) := \rho_y(x_{S^{\rm{c}}})\cdot \mu(\d x) \quad \hbox{for each}\ y.\]
These new measures satisfy $(\mu_y)_{S^{\rm{c}}} = (\nu_y)_{S^{\rm{c}}}$ for every $y$.

Observe that the average $\int \rho_y \,\mu_S(\d y)$ is equal to the Radon--Nikodym derivative of $\int (\nu_y)_{S^{\rm{c}}}\,\mu_S(\d y)$ with respect to $\mu_{S^{\rm{c}}}$, which is identically equal to $1$, by~\eqref{eq:proj-mix}.  This fact and Fubini's theorem give
\[\int_L\mu_y(A)\,\mu_S(\d y) = \int_A\Big(\int_L\rho_y(x_{S^{\rm{c}}})\,\mu_S(\d y)\Big)\  \mu(\d x) = \int_A 1\,\mu(d x) = \mu(A)\]
for all measurable sets $A \subseteq \prod_i K_i$, so we have expressed $\mu$ itself as a mixture:
\begin{equation}\label{eq:mu-mix}
\mu = \int_L \mu_y\,\mu_S(\d y).
\end{equation}

This new mixture has the same mutual information as the mixture in~\eqref{eq:proj-mix}, because by~\eqref{eq:KL-chain5} that mutual information may be written as
\[\int \rmD(\mu_y\,\|\,\mu)\,\mu_S(\d y),\]
and the integrand appearing here is equal to
\begin{multline*}
\rmD(\mu_y\,\|\,\mu) = \int \frac{\d\mu_y}{\d\mu}\log \frac{\d\mu_y}{\d\mu}\,\d\mu = \int \rho_y(x_{S^{\rm{c}}})\log \rho_y(x_{S^{\rm{c}}})\,\mu(\d x) \\ = \int \rho_y(x')\log \rho_y(x')\,\mu_{S^{\rm{c}}}(\d x') = \rmD((\nu_y)_{S^{\rm{c}}}\,\|\,\mu_{S^{\rm{c}}}).
\end{multline*}
So the mutual information in the mixture~\eqref{eq:mu-mix} is at most $\DTC(\mu)$.  On the other hand, since $|S| \leq \delta n$, we have
\begin{equation}\label{eq:A-end}
\int_L\ol{d_n}(\mu_y,\xi_y)\ \mu_S(\d y) \leq \frac{|S|}{n} + \frac{|S^{\rm{c}}|}{n}\int_L\ol{d_{S^{\rm{c}}}}\big((\mu_y)_{S^{\rm{c}}},\xi'_y\big)\ \mu_S(\d y) < 2\delta,
\end{equation}
where the first inequality holds by lifting couplings of $(\mu_y)_{S^{\rm{c}}}$ and $\xi'_y$ arbitrarily to couplings of $\mu_y$ and $\xi_y$ and integrating the inequality
\[d_n(x,x') \le \frac{|S|}{n} + \frac{|S^{\rm{c}}|}{n}d_{S^{\rm{c}}}(x_{S^{\rm{c}}},x'_{S^{\rm{c}}}) \quad \Big(x,x' \in \prod_i K_i\Big).\]
\end{Proof}

\begin{rmk}
The proof of Theorem A obtains the required mixture of near-products by conditioning on a small, carefully-chosen set $S$ of the coordinates.  In this respect, it is a version of the `wringing method' introduced by Dueck~\cite{Due81} for a proof of the strong converse to the coding theorem for the multiple access channel.  See also~\cite[Section 4]{Ahl82} and~\cite[Section III]{Ahl85} for further development of `wringing' in the context of the multiple access channel.  A very similar method is introduced independently in the recent works~\cite{C-OKrzPerZde18,C-OPer19} to decompose certain Gibbs measures arising from factor-graph models into mixtures of approximate `pure states'.  In those works, as in ours, a large collection of random variables with modest dependence is brought close to independence by conditioning on a few of their number.  However, all of those papers obtain only decorrelations among bounded-size subfamilies of the random variables (usually just pairs).  By contrast, we obtain an approximation in $\ol{d_n}$, which is qualitatively stronger because it provides a coupling under which the joint behaviour of all coordinates is controlled at once.  As a result the specific shape of our estimates is different from those predecessors.
	\end{rmk}

\subsection{Bounding number of terms instead of mutual information}\label{subs:number-of-terms}

Theorem A provides a mixture of near-product measures with a bound on the mutual information in the mixture.  We could be slightly more demanding and ask for a bound on the number of terms.

Here we present two results of this kind.  The first is obtained by combining Theorem A with the following lemma about `sampling' from a mixture.

\begin{lem}[Sampling from a mixture with bounded mutual information]\label{lem:sampling}
 Let $(K,\mu)$ be a probability space, and let $\mu$ be represented as a mixture $\int_L \mu_\bullet\,\d \nu$. Let $\eps \in (0,1/2)$, let $L' \subseteq L$ be measurable with $\nu(L') > 1 - \eps/2$, and assume that the mutual information $I:= \rmI(\nu,\mu_\bullet)$ is finite.  Finally, let $m:= \lceil 16\eps^{-2}\rme^{16(I+1)/\eps}\rceil$. Then there exist $y_1$, \dots, $y_m \in L'$ such that
	\begin{equation}\label{eq:TV-close}
	\Big\|\frac{1}{m}\sum_{j=1}^m\mu_{y_j} - \mu\Big\| < 3\eps,
	\end{equation}
where $\|\cdot\|$ denotes the total variation norm.
\end{lem}

This lemma is taken directly from~\cite[Proposition 7.2]{Aus--WP}.  To prove it, the elements $y_1$, \dots, $y_m$ are chosen independently at random from $\nu$ and shown to have the desired property with high probability.  The details are similar to elementary proofs of the law of large numbers: a truncation followed by a variance estimate.  In~\cite{Aus--WP}, the value $I$ in this lemma is written as the $\nu$-integral of the KL divergences $\rmD(\mu_\bullet\,\|\,\mu)$, but that integral is equal to $\rmI(\nu,\mu_\bullet)$ by~\eqref{eq:KL-chain5}. Also, in~\cite[Proposition 7.2]{Aus--WP} it is assumed that $K$ is standard Borel, but a quick check shows that this assumption plays no role in the proof.

By applying Lemma~\ref{lem:sampling} to the mixture in Theorem A, we obtain the following variant of that theorem.

\begin{thm}\label{thm:A2}
In the setting of Theorem A, fix parameters $\eps,\delta > 0$ and let $\mu$ be a probability measure on $\prod_i K_i$.  If $\DTC(\mu) \leq \delta^3 n$, then $\mu$ may be written as a mixture
\[\mu = \frac{1}{m}\mu'_1 + \cdots + \frac{1}{m} \mu'_m\]
so that
\begin{enumerate}
	\item[(a)] $m = O(\eps^{-2})\exp (O(\DTC/\eps))$, and
	\item[(b)] there are product measures $\xi'_1$, \dots, $\xi'_m$ on $\prod_i K_i$ such that
	\[\frac{1}{m}\sum_{j=1}^m\ol{d_n}(\mu'_j,\xi'_j) < 3\eps + 4\delta/\eps.\]
\end{enumerate}
	\end{thm}

\begin{Proof}
	Let
	\[\mu = \int_L \mu_y\,\nu(\d y)\]
	be the representation of $\mu$ promised by Theorem A.  By part (b) of Theorem A and Markov's inequality, the set
	\[L' := \big\{y \in L:\ \ol{d_n}(\mu_y,\xi_y) < 4\delta/\eps\big\}\]
	has $\nu(L') > 1 - \eps/2$.  Therefore Lemma~\ref{lem:sampling} gives $y_1,\dots,y_m \in L'$ such that~\eqref{eq:TV-close} holds, where $m$ satisfies part (a) of the present theorem.
	
	We have not quite found a representation of $\mu$ as a finite mixture yet, but~\eqref{eq:TV-close} says we are close.  Let
	\begin{equation}\label{eq:gamma}
	\g := \frac{1}{m}\sum_{j=1}^m\mu_{y_j}.
	\end{equation}
	Since $\ol{d_n}$ is always bounded by total variation (see, for instance,~\cite[Subsection 4.2]{Aus--WP}), it follows from~\eqref{eq:TV-close} that some coupling $\l$ of $\mu$ and $\g$ satisfies
	\[\int d_n\,\d\l < 3\eps.\]
	For each $j=1,2,\dots,m$, define
	\[\rho_j := \d\mu_{y_j}/\d\g \quad \hbox{and}\quad \l_j(\d x,\d y) := \rho_j(y)\cdot \l(\d x,\d y),\]
	and let $\mu'_j$ be the first marginal of $\l_j$.  Then the average
	\begin{equation}\label{eq:ave}
\frac{1}{m}\sum_{j=1}^m\mu'_j(\d x)
\end{equation}
	is the first marginal of
	\[\frac{1}{m}\sum_{j=1}^m\l_j(\d x,\d y) = \Big(\frac{1}{m}\sum_{j=1}^m\rho_j(y)\Big)\cdot \l(\d x,\d y) = \l(\d x,\d y),\]
	where the second equality holds because of~\eqref{eq:gamma}.  Therefore the average~\eqref{eq:ave} equals $\mu$ exactly.
	
	To finish, let us prove part (b) of the present theorem for these measures $\mu_j'$.  Let $\xi'_j := \xi_{y_j}$ for each $j$.  By construction, $\l_j$ is a coupling of $\mu'_j$ and $\mu_{y_j}$ for each $j$, and therefore
	\begin{align*}
	\frac{1}{m}\sum_{j=1}^m \ol{d_n}(\mu'_j,\xi'_j) &\leq \frac{1}{m}\sum_{j=1}^m \Big[\int d_n\,\d\l_j + \ol{d_n}(\mu_{y_j},\xi_{y_j})\Big] \\
	& = \int d_n\,\d\l + \frac{1}{m}\sum_{j=1}^m\ol{d_n}(\mu_{y_j},\xi_{y_j})\\
	& < 3\eps + 4\delta/\eps,
	\end{align*}
	where the last line follows by our initial choice of $\l$ and the definition of $L'$.
	\end{Proof}

\bigskip

By taking $\eps := \sqrt{\delta}$ in Theorem~\ref{thm:A2}, we obtain $m = O(\delta^{-1})\exp(\DTC(\mu)/\sqrt{\delta})$ and
\[\frac{1}{m}\sum_{j=1}^m\ol{d_n}(\mu'_j,\xi'_j) = O(\sqrt{\delta}).\]
I do not know whether this dependence on $\delta$ can be qualitatively improved.

A different approach to obtaining a finite mixture is available if each $K_i$ is finite, but it gives bounds that depend on their cardinalitites.  Suppose for simplicity that these cardinalities are all at most $k$. Then we can make the following slight adjustment within the proof of Theorem A itself.  Consider again the mixture~\eqref{eq:canon-mix}.  If we simply take this for the mixture produced by the theorem (rather than projecting to $K_{S^{\rm{c}}}$, bounding the mutual information, and lifting back to $K_{[n]}$), then a different bound results.  The number of terms in the mixture~\eqref{eq:canon-mix} is at most
\[|L| = \prod_{i\in S}|K_i| \leq k^{\DTC(\mu)/\delta^2}.\]
Now, for each $y \in L$, let $\xi_y$ be the product measure on $\prod_i K_i$ with the same marginals as $\nu_y$. As in the proof of Theorem A, we have the estimates
\begin{align*}
	\int_L\ol{d_n}\big(\nu_y,\xi_y)\ \mu_S(\d y) &\leq \int_L\sqrt{\frac{1}{2n}\TC(\nu_y)}\ \mu_S(\d y)\\
	& \leq \sqrt{\frac{1}{2n}\TC(X_1\,;\,\dots\,;\,X_n\,|\,X_S)} \leq \delta/\sqrt{2} < \delta.
\end{align*}
Thus we obtain the following alternative to Theorem A.

\begin{thm}\label{thm:A'}
Assume that each $K_i$ has cardinality at most $k$.  Fix a parameter $\delta > 0$.  Then any $\mu \in \Pr(\prod_i K_i)$ may be written as a mixture
	\begin{equation}
		\mu = p_1\mu_1 + \dots + p_m\mu_m
	\end{equation}
	in which
	\begin{enumerate}
		\item[(a)] $m\leq k^{\DTC(\mu)/\delta^2}$,
		\item[(b)] there are product measures $\xi_j$ on $\prod_i K_i$, $j=1,2,\dots,m$, such that
		\[\sum_{j=1}^m p_j\cdot \ol{d_n}(\mu_j,\xi_j) < \delta.\]
	\end{enumerate}
\end{thm}

\begin{rmk}
	Note that Theorem~\ref{thm:A'} can be applied for any value of $\DTC(\mu)$, and gives a nontrivial conclusion provided $\DTC(\mu) \ll \delta^2 n$. This is slightly less restrictive than Theorem A itself, which requires $\DTC(\mu) \leq \delta^3 n$.  The extra power of $\delta$ is needed in the proof of Theorem A because the cardinality $|S|$ shows up in the estimate~\eqref{eq:A-end}, and we need this to be comparable to the integral of $\ol{d_{S^{\rm{c}}}}\big((\mu_y)_{S^{\rm{c}}},\xi'_y)$ in that estimate. Estimate~\eqref{eq:A-end} is not used in the proof of Theorem~\ref{thm:A'}.
\end{rmk}

\section{Directions for further research}

\subsection{`Very small' alphabets}

Theorems A and~\ref{thm:A2} do not depend on $\max_i |K_i|$ --- indeed, they allow infinite alphabets. But Theorem~\ref{thm:A'} does depend on $\max_i |K_i|$.

One can also ask after `very small' alphabets.  This could mean, for instance, that $K_i = \{0,1\}$ for each $i$, and also that $\mu$ is very biased in every coordinate, say
\[\mu\{x_i=1\} \leq p \quad \forall i,\]
where $p \in (0,1)$ is small.

\begin{ques}
 Do our estimates in Part II have any natural refinements if one also allows them to depend on $p$?
	\end{ques}

\subsection{Finer approximations than in transportation}

As discussed in the Introduction, our stability results for TC and DTC are incomplete.  They promise a certain structure only up to a perturbation in $\ol{d_n}$, but neither TC nor DTC is uniformly continuous for that kind of perturbation (TC does satisfy Lemma~\ref{lem:TC-robust}, but that estimate depends on the size of the alphabets).

\begin{ques}
 Is there a finer kind of approximation than $\ol{d_n}$ which enables a complete description of the structure behind small values of TC or DTC?
\end{ques}

The simplest possibility to consider would be approximation in total variation, but I expect this is too strong.

\subsection{Gaps in other entropy inequalities}

The facts that TC and DTC are non-negative can both be subsumed into a much more general inequality of Shearer from~\cite{ChungGraFraShe86}.  If $\cal{S}$ is a family of subsets of $[n]$ with the property that every member of $[n]$ lies in at least $k$ members of $\cal{S}$, then Shearer's inequality asserts that
\begin{equation}\label{eq:Shearer}
\rmH(X_1,\dots,X_n) \leq \frac{1}{k}\sum_{S \in \cal{S}}\rmH(X_i:\ i\in S).
\end{equation}
Han~\cite{Han78} established this previously in case $\cal{S} = {[n]\choose \ell}$ for some $\ell$.  Among these special cases, TC is the gap in~\eqref{eq:Shearer} when $\cal{S} = {[n] \choose 1}$. After a simple application of the chain rule, DTC is the gap in~\eqref{eq:Shearer} when $\cal{S} = {[n]\choose n-1}$ (up to a normalization by $n-1$).  Several other relations among the gaps in Han's inequalities were explored by Fujishige~\cite{Fuj78}.

By now, Shearer's inequality has opened a rich vein of applications in combinatorics: see, for instance,~\cite{Radha01}.  More recently it has been refined further by Madiman and Tetali~\cite{MadTet10} and Balister and Bollob\'as~\cite{BalBol12}.  Madiman and Tetali's paper gives a more careful overview of other work in this direction.  From any of these inequalities, one can define a new notion of multi-variate correlation by considering the gap between the two sides.  Madiman and Tetali also explicitly introduce the gaps in their inequalities in~\cite[Section VII]{MadTet10}, and establish some identities between them, generalizing Fujishige's work.

As far as I know these investigations do not include most of our work from Part I above, let alone Part II.  Even the fact of monotonicity, which generalizes non-negativity, does not seem to be in the literature.

\begin{ques}
 Do the inequalities and identities from Part I have natural generalizations to the gaps in some of the other inequalities mentioned above?
\end{ques}

\begin{ques}
 Do any of those other gaps enjoy stability or structural results of a similar flavour to Theorem A?
\end{ques}

It would be especially interesting if the right choice of gap could be used to capture other structural features than near-products or their low-information mixtures.

All of the entropy inequalities mentioned in the previous subsection are `Shannon inequalities', meaning that they are corollaries of the strong subadditivity of Shannon entropy.  More recently, other `non-Shannon inequalities' have also been investigated: see~\cite{MakMakRomVer02,Mat07,ZhaYeu98} for an introduction to these.  The paper~\cite{FritzCha13} uses both Shannon and non-Shannon inequalities to derive necessary conditions for the existence of a joint probability distribution given some constraints on its marginals, and the paper~\cite{DouFreZeg07} applies some non-Shannon inequalities to a network information theory model.

\begin{ques}
 Do any of these non-Shannon inequalities admit natural stability results?
\end{ques}

\section*{Acknowledgements}
\small

I am grateful to Guido Mont\'ufar for insightful discussions and helpful guidance to the literature.  Two anonymous referees prompted several improvements to the paper and provided further references.

\makesubmdate

\def\cprime{$'$} \def\cprime{$'$}

\makecontacts

\end{document}